\documentclass[11pt,a4paper,reqno]{amsart}
\usepackage{amsfonts}
\usepackage{amsthm}
\usepackage{amsmath}
\usepackage{amscd}
\usepackage[latin2]{inputenc}
\usepackage{t1enc}
\usepackage[mathscr]{eucal}
\usepackage{indentfirst}
\usepackage{graphicx}
\usepackage{graphics}
\usepackage{pict2e}
\usepackage{dsfont}
\usepackage{epic}
\numberwithin{equation}{section}
\usepackage[margin=2.9cm]{geometry}
\usepackage{epstopdf} 
\usepackage{amssymb}
\usepackage{setspace}
\usepackage{enumerate}
\usepackage{bigstrut}
\usepackage{multirow}
\usepackage{mathtools,xparse}
\usepackage{mathrsfs}
\usepackage{hyperref}
\usepackage [autostyle, english = american]{csquotes}
\newtheorem*{lemma*}{Lemma}
\newtheorem{theorem}{Theorem}[section]

\newtheorem{lemma}[theorem]{Lemma}
\newtheorem{proposition}[theorem]{Proposition}
\theoremstyle{definition}
\newtheorem{definition}{Definition}
\newtheorem{remark}[theorem]{Remark}
\newtheorem{assumption}{Assumption}

\DeclareMathOperator*{\argmax}{arg\,max}
\allowdisplaybreaks
\usepackage{chngcntr}

\def\be{\begin{equation}}
\def\ee{\end{equation}}
\def\bepro{\begin{proposition}}
\def\enpro{\end{proposition}}
\def\belemma{\begin{lemma}}
\def\enlemma{\end{lemma}}
\def\it{\textit}

\newcommand{\E}{\mathop{\mathbb{E}}}
\newcommand{\1}{\mathds{1}}

\newcommand{\R}{\mathbb{R}}
\newcommand{\norm}[1]{\left\lVert#1\right\rVert}

\begin{document}

\title[Microcanonical ensembles with multiple constraints]{Large deviations and localization of the microcanonical ensembles given by multiple constraints}

\author{Kyeongsik Nam}

\begin{abstract}
We develop a unified theory to analyze the microcanonical ensembles with several constraints given by unbounded observables. Several interesting phenomena that do not occur in the single constraint case can happen under the multiple constraints case. We systematically analyze the detailed structures of such microcanonical ensembles in two orthogonal directions using the theory of large deviations. First of all, we establish the equivalence of ensembles result, which exhibits an interesting phase transition phenomenon. Secondly, we study the localization and delocalization phenomena  by obtaining large deviation results for the joint law of empirical distributions and the maximum component. Some concrete  examples for which the theory applies will be given as well.
\end{abstract}

\address{ Department of Mathematics, Evans Hall, University of California, Berkeley, CA
94720, USA} 

\email{ksnam@math.berkeley.edu}

 \subjclass[2010]{60D05, 60F10, 82B05, 82B26}

 \keywords{Large deviation principle, microcanonical ensemble, equivalence of ensembles, Gibbs conditioning principle, phase transition, localization} 
 
\maketitle

\section{introduction}
\subsection{Motivation}
There are several notions of the statistical ensembles describing the mechanical system. For instance, a canonical ensemble  represents the possible states in the equilibrium with a heat reservoir at a fixed temperature, whereas a microcanonical ensemble  represents the states having a specified total energy.  The Gibbs' principle, which is also called the principle of  \it{equivalence of ensembles}, states that in the infinite volume limit, the microcanonical ensemble converges to the canonical ensemble with a certain temperature. The theory of \it{large deviations} has provided an elegant way to describe the equivalence of ensembles results.  We refer to \cite{rs} for a monograph about Gibbs measures and  equilibrium statistical mechanics.

The theory of microcanonical  ensembles with a  single constraint has been well-established. In the simplest case when the single constraint is given by
\begin{align} \label{single0}
\left \vert\frac{\phi(X_1)+\cdots+\phi(X_n)}{n}-c\right \vert \leq \delta, \quad  \delta>0 \ \text{small},
\end{align}
or more generally under the presence of an interacting potential, the classical equivalence of ensembles result provides a thermodynamic behavior of the microcanonical ensemble (see  for instance \cite{dsz,g2,g3,g4}). The single constraint \eqref{single0} with the unbounded function $\phi$ has a great  importance since it naturally arises in the various areas of mathematics and physics. In particular, the  microcanonical distributions given by a single $l^p$-constraint ($\phi(x)=x^p$ in \eqref{single0}) have been studied extensively due to its wide applications in geometry and PDE theory. For instance, in \cite{n,nr}, the surface measure and cone measure of the $l^p$-sphere are analyzed, and the annealed, quenched large deviations for the random projection of $l^p$-spheres are established in \cite{gkr}. Also, Barthe et al. \cite{bgmn} provided the probabilistic method to interpret the volume measure of $l^p$-balls. We refer to \cite{df,kr,l2,rr,s} for more details about the $l^p$-spheres.

Beyond the single constraint, it is natural to consider the microcanonical ensembles given by several constraints with unbounded observables. One  motivation for studying these types of ensembles comes from the \it{nonlinear Schr\"odinger equation} (NLS) on $\R^d$:
\begin{align} \label{1}
u_t=-\Delta u + \kappa |u|^{p-1}u.
\end{align}
The NLS \eqref{1} can be regarded as the infinite-dimensional Hamiltonian ordinary differential equation with the Hamiltonian $H$ given by
\begin{align} \label{2}
 H(u)=\int \frac{1}{2}|\nabla u|^2  + \frac{\kappa}{p+1}|u|^{p+1}.
\end{align}
Since $H$ is conserved under the Hamiltonian flow, formally speaking, the Gibbs measure of the form
\begin{align*}
\exp\big(-\beta H(u)\big) dP(u)
\end{align*}
for the fictitious Lebesgue measure $dP$ on the space of functions would be the invariant measure under NLS flow according to the Liouville's theorem. This can be made rigorous when the underlying space is $\mathbb{T}^d$ and  the notion of Gaussian free field is introduced. It has been shown  that the Gibbs measures of the type
\begin{align} \label{3}
\exp\left(\kappa \beta \int_{\mathbb{T}^d} |u|^{p+1} dx\right) dQ(u)
\end{align}
($dQ$ denotes the Gaussian free field on $\mathbb{T}^d$) exist for the certain values of $d$ and $p$ (see \cite{b1,b2,bs,lrs} for details).   The concept of Gibbs measure has played a crucial role in understanding the qualitative properties of certain solutions to NLS. For instance,  invariance of the Gibbs measure \eqref{3} and the probabilistic well-posedness have been established for a large class of equations. We refer to  \cite{d,r,tv} for more details in this direction.

Another way to understand the qualitative behaviors of a solution to NLS is to consider the microcanonical ensembles.   Note that not only the Hamiltonian $H$ defined in \eqref{2} is conserved under the NLS flow \eqref{1}, the mass
\begin{align}\label{mass}
M(u) = \int |u|^2 
\end{align} 
is also  conserved. Motivated by this, the natural invariant  measures of NLS \eqref{1}, considered first by Chatterjee \cite{cha1}, are the conditional distributions of the fictitious Lebesgue measure on the space of functions on the set:
\begin{align*}
|M(u)-m|\leq \delta,\quad |H(u)-E|\leq \delta, \quad \delta>0 \ \text{small}.
\end{align*} 
This can be made rigorous once the underlying space $\R^d$ is discretized into the grids with size $h$. More precisely, for $V_n=\{0,1,\cdots,n-1\}^d$ and the mass and Hamiltonian defined by
\begin{align*}
M_{h,n}(u) = h^d \sum_{x\in V_n} |u(x)|^2,\quad H_{h,n}(u)=\frac{h^d}{2}\sum_{x,y\in V_n, x \sim y} \Big \vert \frac{u(x)-u(y)}{h} \Big \vert ^2 + \frac{\kappa h^d}{p+1}\sum_{x\in V_n} |u(x)|^{p+1}
\end{align*}
($x\sim y$ means that $x$ and $y$ are adjacent), the constraint is defined by
\begin{align} \label{4}
C_{\delta,h,n}:=\{u\in \mathbb{C}^{V_n} |\ |M_{h,n}(u)-m|\leq \delta,\ |H_{h,n}(u)-E|\leq \delta \}.
\end{align} Now, let us consider the uniform distribution on the set \eqref{4}. This  microcanonical ensemble  has an advantage over the Gibbs measures of the form \eqref{3} in the sense that it can be defined for general values of $d$ and $p$.

A solution to the defocusing NLS ($\kappa=1$ in \eqref{1}) exhibits the dispersive behavior like the linear Schr\"odinger equation. On the other hand, in the case of focusing NLS ($\kappa=-1$ in \eqref{1}), the existence of a \it{ground state soliton} $Q_m$ demonstrates that the dispersion may not happen. Note that at the mass-subcritical regime ($1<p<1+\frac{4}{d}$), the ground state soliton  is a unique minimizer (up to the spatial and phase translation) of the variational problem:
 \begin{align*}
E(m):=\inf_{M(f)=m} H(f).
\end{align*}
 The \it{soliton resolution conjecture} claims that in the mass-subcritical regime, for the generic initial conditions, a solution to the focusing NLS gets closer to a soliton. In \cite{cha1}, Chatterjee established the statistical version of the soliton resolution conjecture by studying a thermodynamic limit of the microcanonical ensemble of type \eqref{4}. He showed that when $E(m)<E$, if the discrete function $f_{\delta,h,n}$ is selected randomly according to the uniform distribution on the set $C_{\delta,h,n}$ in \eqref{4}, and $\tilde{f}_{\delta,h,n}$ is denoted by its continuum extension, then for $2<q\leq \infty$ and any $\epsilon>0$,
 \begin{align} \label{soliton}
 \lim_{h\rightarrow 0}\limsup_{\delta \rightarrow 0}\limsup_{n\rightarrow \infty} \mathbb{P}(L^q(\tilde{f}_{\delta,h,n},Q_m)>\epsilon ) =0.
 \end{align}
 Here, 
 $L^q(u,v):=\inf_{\alpha,\beta} \norm{u(\cdot) - e^{i\alpha}v(\cdot + \beta)}_q$.
 In other words, a typical function $f_{\delta,h,n}$ of the microcanonical ensemble \eqref{4} approximates the ground state soliton in a certain sense. 
This interesting result demonstrates that   understanding the structures of  the  microcanonical ensembles with multiple constraints can lead to study qualitative and statistical properties of a solution to the corresponding PDE \eqref{1}. We refer to \cite{t} for a monograph on the general dispersive PDE theory and  the soliton resolution conjecture. 

The remarkable aspect of the statement \eqref{soliton} is that  although the energy of a typical function $f_{\delta,h,n}$ converges to $E$, the energy of a ground state soliton $Q_m$ is $E(m)$ which is strictly less than $E$. This may look a contradiction, but it is plausible since  the $L^q$ norm is too weak to control the Hamiltonian \eqref{2}. Roughly speaking, \eqref{soliton} implies that a typical function $f_{\delta,h,n}$ can be decomposed into  $f^1_{\delta,h,n}$ and $f^2_{\delta,h,n}$ such that $f^1_{\delta,h,n}$ has a negligible $L^\infty$ norm but possesses a strictly positive energy, whereas $f^2_{\delta,h,n}$ is close to the ground state soliton $Q_m$ and has an energy close to $E(m)$ (see \cite{cha1} for details). This striking phenomenon essentially arises from the fact that the microcanonical ensemble \eqref{4} has several constraints with unbounded observables (the mass  \eqref{mass} corresponds to the $l^2$-type constraint $x\mapsto x^2$, and the Hamiltonian \eqref{2} involves the gradient term and a polynomial $x\mapsto x^{p+1}$). 

In order to exemplify  the above phenomenon more concretely, let us first consider the microcanonical ensembles given by two constraints:
\begin{align} \label{two}
\bigcap_{i=1}^2 \Big\{\left \vert\frac{\phi_i(X_1)+\cdots+\phi_i(X_n)}{n}-a_i\right \vert \leq \delta \Big\}, \quad  \delta>0 \ \text{small},
\end{align}
with bounded and continuous $\phi_i$'s. Here, the reference measure on the configuration space $\Omega:=(0,\infty)^\mathbb{N}$ is given by $\mathbb{P}:=\lambda^{\otimes \mathbb{N}}$ for a probability measure $\lambda$ on $(0,\infty)$, and $X_i:\Omega \rightarrow (0,\infty)$ is a projection onto the $i$-th coordinate. The classical Gibbs' principle asserts that  as $n\rightarrow \infty$ followed by $\delta \rightarrow 0$, the law  of $X_1$ converges to the distribution
$d\lambda^* = \frac{1}{Z}e^{\alpha \phi_1 + \beta \phi_2} d\lambda$ for some $\alpha,\beta$ satisfying
$\int \phi_i d\lambda^* = a_i$, $i=1,2$.

However, unlike the microcanonical ensembles \eqref{two} with bounded $\phi_i$'s, several interesting phenomena can happen when $\phi_i$'s are unbounded observables. For instance, when the configuration space is $(0,\infty)^{\mathbb{N}}$, consider the uniform distribution on the set
\begin{align} \label{lp}
\left\{\frac{\phi_1(X_1)+\cdots+\phi_1(X_n)}{n}=1\right \} \bigcap  \left\{\frac{\phi_2(X_1)+\cdots+\phi_2(X_n)}{n}=b\right\}
\end{align}
with unbounded functions $\phi_i(x)=x^i$ for $i=1,2$.
Chatterjee \cite{cha2} established the convergence of the finite marginal distributions  of the microcanonical ensemble \eqref{lp}. When $1\leq b\leq 2$, as $n\rightarrow \infty$, the law of $X_1$ converges weakly to the $G_{1,b}$-distribution, where $G_{1,b}$ is a probability distribution on $(0,\infty)$ of the form $\frac{1}{Z}e^{rx+sx^2}dx$  satisfying 
\begin{align*}
\int xdG_{1,b} = 1, \quad \int  x^2 dG_{1,b} = b.
\end{align*} 
On the other hand,  when $b>2$, as $n\rightarrow \infty$, the law of $X_1$ converges weakly to  $\lambda^*:=\text{exp}(1)$ distribution. The striking fact is that the expectation of $\phi_2$ under $\lambda^*$, which is equal to 2, is strictly less than $b$. In other words, in a thermodynamic limit, the discrepancy $b-2>0$ occurs in the second constraint.  This is what happens in the microcanonical ensemble \eqref{4}. Roughly speaking, $\text{exp}(1)$ distribution plays the role of the ground state soliton $Q_m$  in \eqref{soliton}.  In fact, like the microcanonical ensemble \eqref{4}, in a thermodynamic limit of \eqref{lp} when $b>2$, some localized site $1\leq i\leq n$ would possess a strictly positive $l^2$-mass $\frac{x_i^2}{n}$ and a negligible $l^1$-mass $\frac{x_i}{n}$ due to the existence of a discrepancy $b-2$ corresponding to the second constraint (see \cite{cha2} for details). The example \eqref{lp} shows that the microcanonical ensembles with  several unbounded  constraints  behave qualitatively differently from the microcanonical ensembles \eqref{two} with bounded observables $\phi_i$'s.

 As mentioned earlier, it is crucial to understand the microcanonical ensembles with several constraints,  particularly given by unbounded  observables such as \eqref{4}, since they  have wide applications in  geometry and PDE theory as well as statistical mechanics. However, to the author's knowledge, no systematic and  unified methods to analyze such microcanonical ensembles have  been developed yet. 
  In this paper, we develop a new and unified theory to study the detailed structures of such microcanonical distributions in two orthogonal directions.  Remarkably, these types of microcanonical ensembles turn out to behave  differently from the single constraint case or the multiple constraints case given by bounded  observables.

  \subsection{Previous works and our contributions}

A theory of equivalence between the microcanonical ensemble given by a single constraint and the grand canonical ensemble is quite classical and has been studied extensively. See for example \cite{dsz,lps} for a bounded interacting potential case and \cite{g2,g3,g4} for  a possibly unbounded interacting potential case.  Beyond these classical cases, a  specific kind of the microcanonical ensembles given by several constraints with unbounded observables (see \eqref{4} and \eqref{lp}) is studied in \cite{cha1,cha2}, but the methods used in there are ad hoc and finer structures of the microcanonical ensembles are far from being well-understood.   The first main contribution of our work is to establish the equivalence of ensembles result for general microcanonical ensembles given by several constraints with unbounded observables.  This result is new even for the simplest case such as under the absence of interacting potentials, and surprisingly
such microcanonical ensembles  behave differently from the single constraint case.
 
In order to illustrate this,  let us  consider the  microcanonical ensemble given by the  single constraint: consider the uniform distribution on the set
\begin{align} \label{single}
\Big \{ \left \vert\frac{\phi(X_1)+\cdots+\phi(X_n)}{n}-c\right \vert \leq \delta \Big \}, \quad  \delta>0 \ \text{small}.
\end{align} 
The  \it{maximum entropy principle} asserts that   as $n\rightarrow \infty$ followed by $\delta \rightarrow 0$,  the law of $X_1$ converges to the probability measure $\lambda^*$ maximizing the differential entropy $h(\mu)$ over the constraint $\int \phi d\mu = c$ (see Proposition \ref{prop 5.1} for details). The analogous equivalence of ensembles result is known for the general Hamiltonian with possibly unbounded interacting potentials (see for example \cite[Theorem 3.3]{g4}).

 However, in the case of multiple constraints: 
 \begin{align} \label{several}
\bigcap_{i=1}^k \Big\{\left \vert\frac{\phi_i(X_1)+\cdots+\phi_i(X_n)}{n}-a_i\right \vert \leq \delta \Big\}, \quad  \delta>0 \ \text{small}, \ k\geq 2,
\end{align}
with unbounded functions $\phi_i$'s, an  interesting phenomenon occurs:   some of the  $k$ constraints \eqref{several}
 may become extraneous  in a thermodynamic limit of the microcanonical ensembles. More precisely, under the uniform distribution on the set \eqref{several}, then as $n\rightarrow \infty$ followed by $\delta \rightarrow 0$, the limit distribution of  $X_1$ may be irrelevant to one of the multiple constraints (see Theorem \ref{theorem 1.1} and \ref{theorem 1.0} for details). In other words, as mentioned in the microcanonical ensemble case \eqref{lp}, the limit distribution $\lambda^*$ of $X_1$ may not satisfy $\int \phi_i d\lambda^* = a_i$ for some $i$, which is a striking difference from the single constraint case. We systematically analyze this interesting phenomenon in the first part of the paper as an application of the large deviation theory.

Another remarkable qualitative difference from the single constraint case is the localization phenomenon, which provides the information that  complements the equivalence of ensembles result (see Section \ref{section 2} for the  explanations).   Under the single constraint \eqref{single} with an unbounded function $\phi$, it is not hard to check that localization does not happen (see Proposition \ref{prop 5.2} for a precise statement). However, when the microcanonical ensemble is given by multiple constraints \eqref{several} with unbounded $\phi_i$'s, as mentioned in the example \eqref{lp}, a strictly  positive mass can be concentrated on some sites (see Theorem \ref{theorem 1.4} for details). In the second part of the paper, we systemically study the localization and delocalization phenomena of the microcanonical ensembles using the theory of large deviations. In particular, we derive a large deviation  principle for the joint law of empirical distributions and the maximum component, which reveals a detailed structure of the corresponding microcanonical ensemble.

\subsection{Organization of the paper}
The paper is organized as follows. In Section \ref{section 2}, we state the main theorems and provide their interpretations. In Section \ref{section 3}, we obtain the large deviation results for the joint law of empirical distributions and several empirical means, and then precisely characterize the limit of finite marginal distributions of the microcanonical ensembles. In Section \ref{section 4}, we study the localization and delocalization phenomena of the microcanonical ensembles with multiple constraints.  In Section \ref{section 5}, some concrete examples of the microcanonical distributions for which the theory applies will be covered.

Throughout the paper, for a Polish space $\mathcal{S}$, let us denote $\mathcal{B}$ by the Borel $\sigma$-field on $\mathcal{S}$ and $C_b(\mathcal{S})$ by the set of bounded continuous functions on $\mathcal{S}$. Let us define $\mathcal{M}(\mathcal{S})$ as the set of finite regular Borel measures on $\mathcal{S}$, and $\mathcal{M}_1(\mathcal{S})$ as the subspace of probability measures. Given the set of bounded continuous functions $\{g_k\}$ that determine the weak convergence on $\mathcal{M}_1(\mathcal{S})$, we define a metric  $d$ on $\mathcal{M}_1(\mathcal{S})$ by
\begin{align} \label{metric}
d(\mu,\nu):=\sum_{k=1}^\infty \frac{1}{2^k \norm{g_k}_\infty}\Big[\int_\mathcal{S} g_k d\mu - \int_\mathcal{S} g_k d\nu\Big]
\end{align}
for two probability measures $\mu,\nu$.
Note that the weak topology on  $\mathcal{M}_1(\mathcal{S})$ coincides with the topology given by the metric $d$. Throughout this paper, we assume that $\mathcal{M}_1(S)$ is equipped with the weak topology. Also, $\partial f$ denotes the subdifferential of the function $f$, and we simplify the integral $\int_0^\infty f$ to $\int f$.

\section{Main results} \label{section 2}
Consider the configuration space $\Omega = (0,\infty)^\mathbb{N}$, and let us   denote  $X_i:\Omega \rightarrow (0,\infty)$ by the projection onto the $i$-th coordinate. 
Assume that the functions $\phi_1,\cdots,\phi_k$ $(k\geq 2)$ satisfying the following Assumption \ref{assume0} are given.
\begin{assumption} \label{assume0}
Functions $\phi_1,\cdots, \phi_k$ $(k\geq 2)$ satisfy \\
(C1) For each $1\leq i\leq k$, $\phi_i : (0,\infty)\rightarrow (0,\infty)$ is $C^1$, increasing, and $\lim_{x\rightarrow \infty} \phi_i(x)=\infty$. \\
(C2) For each $1\leq i\leq k$ and any $c>0$, $\int_0^\infty e^{-c\phi_i}dx < \infty$.\\
(C3) There exists $\kappa>1$ such that $\phi_i^\kappa < \phi_{i+1}$ for each $1\leq i\leq k-1$. \\
(C4) There exists $C,M>0$ such that $x>C \Rightarrow \frac{1}{C}\phi_i(x)^{-M} < \phi_i'(x) < C\phi_i(x)^{M}$ for each $i$.
\end{assumption}

Conditions (C1) and (C2) imply that $\phi_i$'s are unbounded and grows not slowly at infinity.
Condition (C3) means that for each index $i$, $\phi_{i+1}$ grows faster than $\phi_i$ at  infinity.  A technical assumption (C4) will be used to prove Lemma \ref{lemma 3.1} later. It is not hard to see that a large class of functions $\phi_i$'s satisfy the Assumption \ref{assume0}. For instance, a large class of polynomials with strictly increasing degrees, which  is of our main interest due to its  wide  applications in geometry and PDEs as explained in the introduction, satisfy the Assumption \ref{assume0}.
In particular, the constraint that Chatterjee considered in \cite{cha2} corresponds to the case $\phi_1(x)=x$ and $\phi_2(x)=x^2$.

For each $1\leq i\leq k$, define the empirical means
\begin{align*}
S^i_n := \frac{\phi_i(X_1)+\cdots+\phi_i(X_n)}{n},
\end{align*} 
 and then consider the following constraints for each $\delta>0$:
\begin{align*}
C^\delta_n := \cap_{i=1}^k \{|S^i_n-a_i|\leq \delta\}.
\end{align*}
We are interested in the infinite volume behavior of the uniform distribution on the constraint $C^\delta_n$ as the gap $\delta$ converges to zero. Since the Lebesgue measure is not a probability measure, we  define a  reference measure $\mathbb{P}$ to be $\mathbb{P}:=\lambda^{\otimes \mathbb{N}}$ on $\Omega = (0,\infty)^{ \mathbb{N}}$, where $\lambda
$ is a probability measure on $(0,\infty)$ defined by
\begin{align} \label{basic}
\lambda = \frac{1}{Z}e^{-\phi_1} dx
\end{align}($Z$ is a normalizing constant). The motivation to choose such reference measure is that it is a probability measure and once conditioned on the constraint $C^\delta_n$, it behaves like the uniform distribution as $\delta \rightarrow 0$. In fact, the conditional distribution of any reference measure $(\frac{1}{Z}e^{p_1\phi_1+\cdots+p_k\phi_k} dx)^{\otimes \mathbb{N}}$  on the constraint $C^\delta_n$ 
approximates the uniform distribution in a certain sense  as $\delta \rightarrow 0$. We refer to Remark \ref{remark 3.2} for the detailed explanations.

Now, let us consider the following microcanonical distribution:
\begin{align}  \label{micro}
\mathbb{P}((X_1,\cdots,X_n)\in \cdot \ |  \ C_n^\delta).
\end{align}
We develop a unifying method to systematically analyze the detailed behaviors of \eqref{micro} as $n\rightarrow \infty$ followed by $\delta \rightarrow 0$.

Note that for certain values of $(a_1,\cdots,a_k)$,  the conditional distribution \eqref{micro} may not be well-defined since the constraint $C^\delta_n$ may be an empty set for small $\delta>0$. In order to avoid this problem, we define the \it{admissible set} in the following way: let us denote $\mathcal{A}_1\subset (0,\infty)^{ (k-1)}$ by
\begin{multline*}
\mathcal{A}_1:=\text{int}\Big\{(v_1,\cdots,v_{k-1})\in (0,\infty)^{ (k-1)}\ \Big \vert \ \exists \mu\in \mathcal{M}_1(\R^+)\ \text{such that} \ h(\mu)\neq -\infty,\\
\int \phi_1 d\mu =v_1, \cdots, \int \phi_{k-1}d\mu=v_{k-1},\int \phi_k d\mu < \infty \Big\}.
\end{multline*}
Here, $h(\mu)$ is the \it{differential entropy}  of the probability measure $\mu\in \mathcal{M}_1(\R^+)$, defined by
\begin{align*}
h(\mu):=&\begin{cases} -\int \frac{d\mu}{dx}\log(\frac{d\mu }{dx}) dx  &\mu \ll dx,\\
-\infty &\text{otherwise}. \end{cases}
\end{align*}
For each $(v_1,\cdots,v_{k-1})\in \mathcal{A}_1$, define 
\begin{align*}
g_1(v_1,\cdots,v_{k-1}):=\inf_{\mu\in \mathcal{M}_1(\R^+)} \Big\{\int \phi_k d\mu \Big \vert h(\mu)\neq -\infty, \int \phi_1 d\mu =v_1, \cdots, \int \phi_{k-1}d\mu=v_{k-1}\Big\}.
\end{align*}
Finally,  the admissible set $\mathcal{A}$ is defined by
\begin{align*}
\mathcal{A}:=  \{(v_1,\cdots,v_{k-1},v_k) | (v_1,\cdots,v_{k-1})\in \mathcal{A}_1, v_k>g_1(v_1,\cdots,v_{k-1})\}.
\end{align*}
Also, we assume that a map $g_1 : \mathcal{A}_1 \rightarrow \R$ is continuous, which implies that $\mathcal{A}$ is an open set.
Throughout this paper, we only consider the case $(a_1,\cdots,a_k)\in \mathcal{A}$ so that the constraint $C^\delta_n$ is a non-empty set, and thus the microcanonical distribution is well-defined (see Remark  \ref{remark 2.5} for the explanations).

We first characterize the law to which the finite marginal distribution $\mathbb{P}((X_1,\cdots,X_j)\in \cdot | C^\delta_n)$ weakly converges as $n\rightarrow \infty$ followed by $\delta \rightarrow 0$. 
\begin{theorem} \label{theorem 1.1} Let $\lambda^*$ be the (unique) maximizer of the differential entropy $h(\cdot)$ over the set
\begin{align} \label{11}
\Big\{\mu \in \mathcal{M}_1(\R^+) \ \Big \vert \ \int \phi_1 d\mu =a_1, \cdots, \int \phi_{k-1}d\mu=a_{k-1}, \int \phi_kd\mu\leq a_k \Big\}.
\end{align}
Then, for any fixed positive integer $j$,
\begin{align} \label{12}
\lim_{\delta \rightarrow 0} \lim_{n \rightarrow \infty} \mathbb{P}((X_1,\cdots,X_j)\in \cdot \ | \ C^\delta_n) = (\lambda^*)^{\otimes j}.
\end{align}
\end{theorem}

\begin{remark} \label{remark 2.2}
When each function $\phi_i$ is bounded and continuous, as a simple application of the maximum entropy principle, one can deduce that the limiting law $\lambda^*$ in \eqref{12} is a (unique) maximizer of the differential entropy $h(\cdot)$ over the set
\begin{align} \label{bounded}
\Big\{\mu \in \mathcal{M}_1(\R^+) \ \Big \vert \ \int \phi_1 d\mu =a_1, \cdots, \int \phi_{k-1}d\mu=a_{k-1}, \int \phi_kd\mu= a_k \Big\}.
\end{align} 
In fact, according to the maximum entropy principle, the limiting distribution $\lambda^*$ in \eqref{12} is a (unique) minimizer of the relative entropy $H(\cdot | \lambda)$ over the set  \eqref{bounded}. Thus, using the identity: for $\mu  \ll dx$,
\begin{align*}
H(\mu | \lambda ) &= \int \log(\frac{d\mu }{d\lambda})d\mu = \int \log(\frac{d\mu }{dx})d\mu +\int \log(\frac{dx}{d\lambda})d\mu  \nonumber \\
&= -h(\mu) + \int \phi_1 d\mu + C = -h(\mu) + a_1 + C,
\end{align*}
it follows that $\lambda^*$ is a (unique) maximizer of the differential entropy $h(\cdot)$ over the set  \eqref{bounded}. 

On the other hand,  when the macroscopic observables $\phi_i$'s are unbounded, the classical maximum entropy principle is not applicable since the map $\mu \mapsto \int \phi_i d\mu$ may not be continuous.  Theorem \ref{theorem 1.1} claims that the last condition $\int \phi_k d\mu = a_k$ in the set \eqref{bounded} is enlarged to the condition $\int \phi_k d\mu \leq a_k$. This implies that for certain values of $a_1,\cdots,a_{k-1}$, the last constraint $|S^k_n-a_k|\leq \delta$ may be irrelevant to the limiting law of the finite marginal distribution of \eqref{micro}. In other words, unlike the case when $\phi_i$'s are bounded continuous, the limiting law $\lambda^*$  may not satisfy $\int \phi_k d\lambda^* = a_k$. This phenomenon is precisely described in Theorem \ref{theorem 1.0}, which is about the equivalence of ensembles result.
\end{remark}

It turns out that as in Remark \ref{remark 2.2}, the structure of a set  \eqref{11} in Theorem \ref{theorem 1.1} is also  different from the case when the microcanonical distribution is given by a  single constraint. In fact, in the case of single constraint \eqref{single} under the reference measure $(\frac{1}{Z}e^{-\phi}dx)^{\otimes \mathbb{N}}$ (assume that a  unbounded function $\phi$ satisfies the conditions (C1) and (C2) in Assumption \ref{assume0}), $\lambda^*$ in Theorem \ref{theorem 1.1} is given by
\begin{align*}
\lambda^* =  \argmax_{\mu \in \mathcal{M}_1(\R^+)} \Big\{h(\mu) \ \Big \vert \ \int \phi d\mu \leq c \Big\} = \argmax_{\mu \in \mathcal{M}_1(\R^+)} \Big\{h(\mu) \ \Big \vert \ \int \phi d\mu =c \Big\}
\end{align*}
(see  Section \ref{section 5.1}  and the identity \eqref{single1}). In other words, even when $\phi$ is unbounded, the limiting distribution $\lambda^*$ satisfies $\int \phi d\lambda^* = c$. We refer to \cite{g3,g4} for the similar equivalence of ensembles result for more general Hamiltonian with superstable interaction.

 On the other hand, as mentioned before, in the case of multiple constraints with unbounded $\phi_i$'s satisfying Assumption \ref{assume0}, the expectation of $\phi_k$ under the limiting distribution $\lambda^*$ may not be equal to $a_k$. Also, the expectation of $\phi_i$'s ($1\leq i\leq k-1$) under the limiting distribution $\lambda^*$ is always equal to $a_i$. This is because the unbounded function $\phi_k$ controls other functions, and the main reason behind this phenomenon is illustrated in Theorem \ref{theorem 2.2}.

Now, let us precisely characterize a unique maximizer of the differential entropy $h(\cdot)$ over the set \eqref{11}. In order to accomplish this, we need the following definition:

\begin{definition} \label{def}
Define the logarithmic moment generating function:
\begin{align}\label{ham}
H(p_1,\cdots,p_k):= \log \int e^{p_1\phi_1+\cdots+p_k\phi_k} d\lambda.
\end{align}
Let us denote $\pi_1$ and $\pi_2$ by the projections $\pi_1(v_1,\cdots,v_{k-1},v_k) = (v_1,\cdots,v_{k-1})$ and $\pi_2(v_1,\cdots,v_{k-1},v_k)=v_k$.
Then, define $\mathcal{S}_1\subset \mathcal{A}_1$ by a collection of $(v_1,\cdots,v_{k-1})$'s such that there exist $p_1,\cdots,p_{k-1}$  satisfying
\begin{align}\label{a1}
(v_1,\cdots,v_{k-1}) \in \pi_1(\partial H(p_1,\cdots,p_{k-1},0)).
\end{align}
For $(v_1,\cdots,v_{k-1})\in \mathcal{S}_1$, choose a unique $(p_1,\cdots,p_{k-1})$ satisfying \eqref{a1} (see Remark \ref{remark 2.9} for the explanations), and then define a function $g_2:\mathcal{S}_1\rightarrow \R$ by
\begin{align} \label{def1}
g_2(v_1,\cdots,v_{k-1}) := \inf \{\pi_2 (\partial H(p_1,\cdots,p_{k-1},0))\}.
\end{align}
 Finally, define $\mathcal{S}_2 := \mathcal{A}_1\cap \mathcal{S}_1^c$. 
\end{definition}

Now, one can precisely characterize the distribution $\lambda^*$  in Theorem \ref{theorem 1.1} using the notions in Definition \ref{def}. It exhibits  an interesting phase transition phenomenon:

\begin{theorem} \label{theorem 1.0}
Fix any positive integer $j$. Then,  \begin{align*}
\lim_{\delta \rightarrow 0} \lim_{n \rightarrow \infty} \mathbb{P}((X_1,\cdots,X_j)\in \cdot \ | \ C^\delta_n) = (\lambda^*)^{\otimes j},
\end{align*}
where $\lambda^*$ is characterized as follows: when $(a_1,\cdots,a_{k-1})\in \mathcal{S}_1$ and $a_k \geq g_2(a_1,\cdots,a_{k-1})$,
\begin{align*}
\lambda^* = \frac{1}{Z}e^{p_1\phi_1+\cdots+p_{k-1}\phi_{k-1}} dx
\end{align*}
for $p_1,\cdots,p_{k-1}$ satisfying $\int \phi_i d\lambda^* = a_i$ for $1\leq i\leq k-1$.

On the other hand, either in the case of \\
(i) $(a_1,\cdots,a_{k-1})\in \mathcal{S}_2$ or \\
(ii) $(a_1,\cdots,a_{k-1})\in \mathcal{S}_1$ and $a_k < g_2(a_1,\cdots,a_{k-1})$,
\begin{align*}
\lambda^*=\frac{1}{Z}e^{p_1\phi_1+\cdots+p_{k-1}\phi_{k-1}+p_k\phi_k} dx
\end{align*}
for $p_1,\cdots,p_{k-1},p_k$ satisfying $p_k<0$ and  $\int \phi_i d\lambda^* = a_i$ for $1\leq i\leq k$.
\end{theorem}
According to the Gibbs' principle, if each $\phi_i$ is bounded continuous, then the limiting law  is of the form
$\lambda^*=\frac{1}{Z}e^{p_1\phi_1+\cdots+p_k\phi_k} dx$
satisfying $\int \phi_i d\lambda^* = a_i$ for all $1\leq i\leq k$. Also, when the microcanonical ensemble is given by a single constraint \eqref{single}, even when $\phi$ is not bounded, one can prove a similar result (see Proposition \ref{prop 5.1}). However,  when the constraints are given by several unbounded observables satisfying  Assumption \ref{assume0}, Theorem \ref{theorem 1.0} demonstrates that one of the constraints may not contribute to the limiting distribution $\lambda^*$. We refer to Section \ref{section 5.2} and \ref{section 5.3} for some concrete examples.

Theorem \ref{theorem 1.0} also shows that the interesting phase transition phenomenon happens in the equivalence of ensembles viewpoint. Indeed, when $(a_1,\cdots,a_{k-1})\in \mathcal{S}_1$ and $a_k \geq g_2(a_1,\cdots,a_{k-1})$, the $k$-th constraint  $S^k_n = a_k$ becomes extraneous for a   limit of the finite marginal distributions of the microcanonical ensembles.  Since  $\lambda^*$ in Theorem \ref{theorem 1.0} satisfies $\int \phi_k d\lambda^* = g_2(a_1,\cdots,a_{k-1})$  (see Lemma \ref{prop 2.9}), it is plausible to guess that the discrepancy $a_k-g_2(a_1,\cdots,a_{k-1})$ corresponding to the $k$-th constraint gets concentrated on some  sites. We will rigorously elaborate on this point in  Theorem \ref{theorem 1.4}.

On the other hand, when $(a_1,\cdots,a_{k-1})\in \mathcal{S}_1$ and $a_k < g_2(a_1,\cdots,a_{k-1})$, in the equivalence of ensembles viewpoint Theorem \ref{theorem 1.0}, the microcanonical distributions \eqref{micro} behave in a standard way. In other words, as in the case when $\phi_i$'s are bounded, the limiting distribution $\lambda^*$  satisfies $\int \phi_i d\lambda^* = a_i$ for all $1\leq i\leq k$. From this, we can infer that no huge amount of the quantity can be concentrated on some  sites (see Theorem \ref{theorem 1.3} for the precise statement), unlike the  case $a_k \geq g_2(a_1,\cdots,a_{k-1})$. Another interesting point of Theorem \ref{theorem 1.0} is the case when  $(a_1,\cdots,a_{k-1})\in \mathcal{S}_2$: unlike the case $(a_1,\cdots,a_{k-1})\in \mathcal{S}_1$,  whatever $a_k$ is, the limiting distribution $\lambda^*$ satisfies $\int \phi_i d\lambda^* = a_i$ for all $1\leq i\leq k$.  This key difference of the sets $\mathcal{S}_1$ and $\mathcal{S}_2$ follows from Lemma \ref{prop 2.9}. 

Although Theorem \ref{theorem 1.0} provides the  equivalence of ensembles result and explains the interesting phase transition phenomenon, it does not  capture the localization phenomenon. In order to illustrate this, assume for a moment that in a thermodynamic limit, a huge amount of the quantity gets concentrated on a single site. It is obvious that the probability that this localized site is  the first coordinate of the configuration space is equal to $\frac{1}{n}$. Since $\frac{1}{n}$ converges to zero as  $n\rightarrow \infty$, this localization phenomenon is not reflected in the  statement:
\begin{align*}
\lim_{\delta \rightarrow 0} \lim_{n \rightarrow \infty} \mathbb{P}(X_1\in \cdot \ | \ C^\delta_n) = \lambda^*.
\end{align*}

Therefore, the  localization and delocalization phenomena can  provide the supplementary information about the microcanonical ensembles. We study this phenomenon by obtaining a large deviation result for the maximum component. In fact, we can analyze much finer structures of \eqref{micro} by establishing a large deviation result for the joint law of empirical distributions $L_n := \frac{1}{n}(\delta_{X_1}+\cdots+\delta_{X_n})$ and the maximum component $M_n := \max_{1\leq i\leq n} \frac{\phi_k(X_i)}{n}$ under the microcanonical distribution \eqref{micro}. 
\begin{theorem} \label{theorem 1.2}
For any Borel set $A$ in  $\mathcal{M}_1(\R^+) \times \R^+$, 
\begin{multline*}
 -\inf_{(\mu,z)\in A^{\mathrm{o}}} J^{max}(\mu, z)\leq  \liminf_{\delta \rightarrow 0} \liminf_{n \rightarrow \infty} \frac{1}{n}\log \mathbb{P}((L_n, M_n)\in A^{\mathrm{o}}|C_n^\delta) \\ \leq  \limsup_{\delta \rightarrow 0} \limsup_{n \rightarrow \infty} \frac{1}{n}\log \mathbb{P}((L_n, M_n)\in \bar{A}|C_n^\delta) \leq -\inf_{(\mu,z)\in \bar{A}} J^{max}(\mu, z),
\end{multline*}
with the rate function $J^{max}(\mu,z)$ given by
\begin{align*}
&J^{max}(\mu,z) \\
=&\begin{cases} -h(\mu) - K(a_1,\cdots,a_k) &\text{if}\ \int \phi_1d\mu=a_1, \cdots, \int \phi_{k-1}d\mu=a_{k-1}, \int \phi_kd\mu \leq a_k-z ,\\
\infty &\text{otherwise}. \end{cases}
\end{align*}
Here,  $K(a_1,\cdots,a_k)$ is defined by
\begin{align*}
K(a_1,\cdots,a_k) = \inf_{\mu\in \mathcal{M}_1(\R^+)} \Big\{-h(\mu)\ \Big \vert \int \phi_1d\mu=a_1,\cdots,\int \phi_{k-1}d\mu=a_{k-1},\int \phi_kd\mu\leq a_k \Big\}.
\end{align*}
\end{theorem}
Theorem  \ref{theorem 1.2} provides  fine structures of the microcanonical ensembles since it offers the limit behaviors of the joint law of empirical distributions and the maximum component. In particular,  one can systematically analyze the  localization and delocalization phenomena of the  microcanonical ensembles  using the large deviation result  Theorem \ref{theorem 1.2}. 
First, one can prove the following delocalization result:

\begin{theorem} \label{theorem 1.3}
 Fix any $\epsilon>0$. Then, either in the case of \\
(i) $(a_1,\cdots,a_{k-1})\in \mathcal{S}_2$ or \\
(ii)$(a_1,\cdots,a_{k-1})\in \mathcal{S}_1$ and $a_k\leq g_2(a_1,\cdots,a_{k-1})$, 
\begin{align} \label{delocal0}
\limsup_{\delta \rightarrow 0} &\limsup_{n\rightarrow \infty} \frac{1}{n}\log \mathbb{P}(M_n \geq \epsilon \ | \  C^\delta_n) < 0.
\end{align}
In particular, localization does not happen in the sense that
\begin{align} \label{delocal}
\lim_{\delta \rightarrow 0} \lim_{n\rightarrow \infty} \mathbb{P}(M_n < \epsilon \ | \ C_n^\delta) = 1.
\end{align}
On the other hand, in the case of $(a_1,\cdots,a_{k-1})\in \mathcal{S}_1$ and $a_k> g_2(a_1,\cdots,a_{k-1})$,  we have the upper tail estimate for the maximum component:
\begin{align} \label{upper}
\limsup_{\delta \rightarrow 0} &\limsup_{n\rightarrow \infty} \frac{1}{n}\log \mathbb{P}(M_n \geq a_k-g_2(a_1,\cdots,a_{k-1})+\epsilon \ | \ C^\delta_n) < 0.
\end{align}
In particular, the maximum component cannot be too large in the sense  that
\begin{align*} 
\limsup_{\delta \rightarrow 0} &\limsup_{n\rightarrow \infty}  \mathbb{P}(M_n < a_k-g_2(a_1,\cdots,a_{k-1})+\epsilon \ | \ C^\delta_n) =1.
\end{align*}
\end{theorem}

Theorem \ref{theorem 1.3} claims that for certain values of $(a_1,\cdots,a_k)$ (condition (i) or (ii) in Theorem \ref{theorem 1.3}), delocalization happens in the sense that \eqref{delocal} holds. On the other hand, when $(a_1,\cdots,a_{k-1})\in \mathcal{S}_1$ and $a_k> g_2(a_1,\cdots,a_{k-1})$, as we predicted before, it is plausible to expect that the localization phenomenon happens. Since Theorem \ref{theorem 1.3} provides  the upper tail estimate \eqref{upper} for the maximum component, if we have an analogous lower tail estimate:
\begin{align}  \label{lower}
\limsup_{\delta \rightarrow 0} \limsup_{n\rightarrow \infty} \frac{1}{n}\log \mathbb{P}(M_n \leq a_k-g_2(a_1,\cdots,a_{k-1})-\epsilon \ | \ C^\delta_n)<0,
\end{align}
then  we can deduce that  $M_n$ approximates to $a_k- g_2(a_1,\cdots,a_{k-1})$ as $n\rightarrow \infty$ followed by $\delta \rightarrow 0$, which implies the localization phenomenon.  Unfortunately, using  the large deviation result Theorem \ref{theorem 1.2}, one can check that
\eqref{lower} is false in general: indeed, \begin{align} \label{lower2}
\lim_{\delta \rightarrow 0} \lim_{n\rightarrow \infty} \frac{1}{n}\log \mathbb{P}(M_n \leq a_k-g_2(a_1,\cdots,a_{k-1})-\epsilon \ | \ C^\delta_n) = 0
\end{align} (see Section \ref{section 4.2} for the explanations). Therefore, in order to obtain the lower tail estimate of type \eqref{lower}, we need to scale down the scaling factor $n$. Remarkably, it turns out that unlike the upper tail estimate \eqref{upper} or the delocalization estimate \eqref{delocal0}, the correct scaling factor in \eqref{lower} highly depends on the detailed structures of the functions $\phi_i$'s: 
\begin{align}  \label{lower1}
\limsup_{\delta \rightarrow 0} \limsup_{n\rightarrow \infty} \frac{1}{g(n)}\log \mathbb{P}(M_n \leq a_k-g_2(a_1,\cdots,a_{k-1})-\epsilon \ | \ C^\delta_n)<0,
\end{align}
for some function $g$ heavily relying on $\phi_i$'s.
 Also, since the scaling factor $g(n)$  grows slowly than $n$, unlike the estimate \eqref{delocal0} or \eqref{upper}, the left hand side of the lower tail estimate \eqref{lower1}  is sensitive to the  particular choice of the reference measure of the form
 \begin{align*}
 \big(\frac{1}{Z}e^{p_1\phi_1+\cdots+p_k\phi_k} dx\big)^{\otimes \mathbb{N}}
\end{align*}  due to the following Remark \ref{remark 3.2}:

\begin{remark} \label{remark 3.2}
We have developed theories under the particular reference measure $\mathbb{P}=\lambda^{\otimes \mathbb{N}}$ with  $\lambda$ given by \eqref{basic} since it is a probability measure  and once conditioned on the constraint $C^\delta_n$, it behaves like the uniform distribution, which is of our main interest. In order to explain this  rigorously, let us consider the probability measure $\nu$ on $(0,\infty)$ given by
\begin{align} \label{320}
\nu = \frac{1}{Z}e^{p_1\phi_1+\cdots+p_k\phi_k} dx.
\end{align}
Then, one can check that for any $n\in \mathbb{N}$, $\delta>0$, and Borel set $A$ in $(\R^+)^n$, 
\begin{align} \label{eq}
e^{-2n(|p_1|+\cdots+|p_k|)\delta}\frac{\text{Leb}(A\cap C^{\delta}_n)}{\text{Leb}(C^{\delta}_n)}\leq \nu^{\otimes n}(A | C^{\delta}_n) \leq e^{2n(|p_1|+\cdots+|p_k|)\delta} \frac{\text{Leb}(A\cap C^{\delta}_n)}{\text{Leb}(C^{\delta}_n)}.
\end{align}
Thus, for any probability measure $\nu$ of the form \eqref{320},
\begin{align} \label{321}
 \limsup_{\delta \rightarrow 0} \limsup_{n\rightarrow \infty} &\frac{1}{n}\log \nu^{\otimes \mathbb{N}}(A|C_n^\delta) = \limsup_{\delta \rightarrow 0} \limsup_{n\rightarrow \infty} \frac{1}{n}\log (\text{Leb})^{\otimes \mathbb{N}}(A|C_n^\delta),
\end{align}
and
\begin{align*}
\liminf_{\delta \rightarrow 0} \liminf_{n\rightarrow \infty} &\frac{1}{n}\log \nu^{\otimes \mathbb{N}} (A|C_n^\delta) =\liminf_{\delta \rightarrow 0} \liminf_{n\rightarrow \infty} \frac{1}{n}\log (\text{Leb})^{\otimes \mathbb{N}}(A|C_n^\delta).
\end{align*}
This implies that Theorem \ref{theorem 1.2} and  \ref{theorem 1.3} hold under  general reference measures of type \eqref{320}, particularly under the uniform distribution  which is of our main interest. Also, the limiting law of the finite marginal distributions of \eqref{micro}  are identical under any reference measures of type \eqref{320} (see Remark \ref{uniform}).

On the other hand, the lower tail estimate of type  \eqref{lower1} depends on the particular choice of the reference measure \eqref{320}. This is because the scaling factor $g(n)$ in \eqref{lower1} grows slower than $n$ at infinity.  In fact, if we switch the reference measure from $(\nu_1)^{\otimes \mathbb{N}}$ to $(\nu_2)^{\otimes \mathbb{N}}$ for $\nu_1$ and $\nu_2$ of the form \eqref{320}, then the cost arising from  this change is  $\mathcal{O}(e^{Cn\delta})$ in the sense that for some constant $C$,
\begin{align*}
e^{-nC\delta}\nu_2^{\otimes n}(A | C^{\delta}_n)\leq \nu_1^{\otimes n}(A | C^{\delta}_n) \leq e^{nC\delta} \nu_2^{\otimes n}(A | C^{\delta}_n).
\end{align*}
 Since the scaling factor $g(n)$ grows slowly than $n$ at infinity, for any fixed $\delta>0$,
\begin{align*}
\lim_{n\rightarrow \infty} \frac{1}{g(n)}\log(e^{Cn\delta}) = \infty.
\end{align*}
This implies that the left hand side of the lower tail estimate \eqref{lower1}  is sensitive to the particular choice of the reference measure of the form \eqref{320}.
\end{remark}

Now, let us study the localization phenomenon by establishing the  lower tail estimate \eqref{lower1}.  Since we have already proved   in Theorem \ref{theorem 1.3} that localization does not happen when  $(a_1,\cdots,a_{k-1})\in \mathcal{S}_2$, we only consider the case  $(a_1,\cdots,a_{k-1})\in \mathcal{S}_1$. Then, one can choose a (unique) probability measure $\nu$ on $(0,\infty)$ of the form
\begin{align} \label{general}
\nu = \frac{1}{Z}e^{p_1\phi_1+\cdots+p_{k-1}\phi_{k-1}} dx
\end{align}
($Z$ is a normalizing constant) satisfying
\begin{align*}
\int \phi_1 d\nu = a_1,\cdots,\int \phi_{k-1} d\nu=a_{k-1}
\end{align*}
(see Lemma \ref{prop 2.9} for the explanations). Note that  $\nu=\lambda^*$, which is the limiting distribution in Theorem \ref{theorem 1.0} when $a_k\geq g_2(a_1,\cdots,a_{k-1})$, satisfies this condition. Let us denote $1\leq m\leq k-1$ by the largest index such that $p_m \neq 0$. As explained in Remark \ref{remark 3.2},  the lower tail estimate \eqref{lower1} depends on the particular choice of the reference measure of form \eqref{320}, and we will establish it under the reference measure $\mathbb{Q} := \nu^{\otimes \mathbb{N}}$.

The reason why we consider such reference measure  to establish the lower tail estimate \eqref{lower1} is as follows. For the probability measure $\mu$ of the form  \eqref{general}, let us denote $I^\mu$ by  the (weak) large deviation rate function for the sequence $(S^1_n,\cdots,S^k_n)$ under  $\mu^{\otimes \mathbb{N}}$. Then,  when $(a_1,\cdots,a_{k-1})\in \mathcal{S}_1$ and $a_k> g_2(a_1,\cdots,a_{k-1})$, due to the estimate \eqref{lower2} and Remark \ref{remark 3.2},  
\begin{gather*}
 \mu^{\otimes \mathbb{N}} (C_n^{\delta})=e^{-nI^\mu(a_1,\cdots,a_k) + r_1(n,\delta)},\\
\mu^{\otimes \mathbb{N}} (\{M_n<a_k-g_2(a_1,\cdots,a_{k-1})-\epsilon \} \cap  C_n^{\delta})=e^{-nI^\mu(a_1,\cdots,a_k) + r_2(n,\delta)}
\end{gather*}
for $r_1(n,\delta)$, $r_2(n,\delta)$ satisfying  
$\lim_{\delta \rightarrow 0} \lim_{n\rightarrow \infty} \frac{r_i(n,\delta)}{n} = 0$ for $i=1,2$.  In order to establish the lower tail estimate of type \eqref{lower1}, we need to analyze  the lower order terms $r_1(n,\delta)$, $r_2(n,\delta)$ since the scaling factor $g(n)$ grows slowly than $n$. Since  the standard large deviation result does not reveal the finer behavior of $r_1(n,\delta)$ and $r_2(n,\delta)$, in order to capture  this detailed structure we choose a probability measure $\mu$ such that $I^\mu(a_1,\cdots,a_k)=0$. Since the probability measure $\nu$ chosen above satisfies $\int \phi_k d\nu = g_2(a_1,\cdots,a_{k-1})$ (see Lemma \ref{prop 2.9}), according to the law of large numbers and the estimate \eqref{301} in Lemma \ref{lemma 3.1},  $I^\nu(a_1,\cdots,a_k) = 0$ whenever $a_k>g_2(a_1,\cdots,a_{k-1})$.

As mentioned before, unlike the  upper tail estimate \eqref{upper} or the delocalization estimate \eqref{delocal0}, the  scaling factor in the lower tail estimate \eqref{lower1} heavily relies on the structures of  functions $\phi_i$'s in a complicated way. Roughly speaking,  for a large class of functions $\phi_i$'s satisfying some technical conditions, the lower tail estimate \eqref{lower1} holds with the scaling factor $g(n):=(\phi_m \circ \phi_k^{-1})(n)$.  We prove this in the particular case when $g(n)$ grows as $n^\gamma$ ($0<\gamma<1$)  for the following two reasons: first of all, we try to keep arguments as simple as possible in order to separate the key ideas of the proof from  technical details. Secondly, when $\phi_i$'s are polynomials, which is of our main interest due to its broad applications in geometry and PDE theory,  $g(n) \approx n^\gamma$ for some $0<\gamma<1$. The following theorem provides the lower tail estimate for the maximum component, and describes the localization phenomenon:

\begin{theorem} \label{theorem 1.4}
Suppose that $(a_1,\cdots,a_{k-1})\in \mathcal{S}_1$ and $a_k>g_2(a_1,\cdots,a_{k-1})$. Assume further that there exist $0<\gamma_1,\cdots,\gamma_{k-1}<1$ such that for each $1\leq i\leq k-1$,
\begin{align} \label{assume2}
 \lim_{x\rightarrow \infty} \frac{(\phi_i\circ \phi_k^{-1})(x) }{x^{\gamma_i}}=1.
\end{align} If the reference measure $\mathbb{Q}$ and the index $m$ are chosen as above, then for any $\epsilon>0$,
\begin{align} \label{141}
\limsup_{\delta \rightarrow 0} \limsup_{n\rightarrow \infty} \frac{1}{n^{\gamma_m}}\log \mathbb{Q}(M_n<a_k-g_2(a_1,\cdots,a_{k-1})-\epsilon \ | \ C_n^{\delta})  <0.
\end{align}
In particular, localization happens in the sense that for any $\epsilon>0$,
\begin{align} \label{142}
\lim_{\delta \rightarrow 0}\lim_{n\rightarrow \infty} \mathbb{Q}(|M_n-(a_k-g_2(a_1,\cdots,a_{k-1}))|<\epsilon \ | \ C_n^{\delta})=1.
\end{align}
\end{theorem}

Since the scaling factor in the lower tail estimate  \eqref{141} grows slowly than $n$, we need a completely different approach from the standard large deviation theory to prove the estimate \eqref{141}. In order to accomplish this, we partially adapt the method used in \cite{cha2}. As mentioned before, following the proof of Theorem \ref{theorem 1.4}, one can check that for a large class of functions $\phi_i$'s satisfying some technical assumptions, the lower tail estimate \eqref{141} with the scaling factor $g(n):=(\phi_m \circ \phi_k^{-1})(n)$ holds as well.

It is important to note that Theorem \ref{theorem 1.3} and Theorem \ref{theorem 1.4}  provide a complete picture of the localization and delocalization phenomena of the microcanonical ensembles with multiple constraints. In fact, let us assume that $(a_1,\cdots,a_{k-1})\in \mathcal{S}_1$, and take the  corresponding reference measure $\mathbb{Q}$ as in Theorem \ref{theorem 1.4}. If $a_k>g_2(a_1,\cdots,a_{k-1})$, then the localization happens in the sense of \eqref{142}, and the delocalization happens at all of the other sites (see Theorem \ref{theorem 3.6} for details). 
On the other hand, if   $a_k\leq g_2(a_1,\cdots,a_{k-1})$, then localization phenomenon does not happen according to Theorem \ref{theorem 1.3} and Remark \ref{remark 3.2}. Note that as mentioned before,  when $(a_1,\cdots,a_{k-1})\in \mathcal{S}_2$, whatever $a_k$ is,  localization phenomenon does not occur (see Theorem \ref{theorem 1.3}).

It is also crucial to note that when the localization happens,   the maximum component $M_n$ behaves  differently in the upper tail and lower tail regime. In fact, the upper tail estimate  is universal in the sense that the estimate \eqref{upper} holds with the scaling factor $n$   for any functions $\phi_i$'s satisfying Assumption \ref{assume0}. On the other hand, the lower tail estimate \eqref{141} is not universal in the sense that the scaling factor heavily relies on the structures of  functions $\phi_i$'s.
In the case when the localization does not happen (condition (i) or (ii) in Theorem \ref{theorem 1.3}), the delocalization estimate \eqref{delocal0} is  universal.

\section{Large deviations  and equivalence of ensembles results} \label{section 3}

In this section, we characterize the limit distribution to which the finite marginal distribution of \eqref{micro} converges. As explained in Section \ref{section 2}, it exhibits a  phase transition phenomenon.  In Section \ref{section 3.1}, we briefly review the theory of large deviations and the classical equivalence of ensembles result. In Section \ref{section 3.2} and \ref{section 3.3}, we prove Theorem \ref{theorem 1.1} using a large deviation theory. In Section \ref{section 3.4}, we precisely characterize the limit distribution $\lambda^*$ in Theorem \ref{theorem 1.1} and conclude the proof of Theorem \ref{theorem 1.0}. Finally, in Section \ref{section 3.5}, we study a structure of the large deviation rate function for  several empirical means.

\subsection{Preliminaries : large deviation principle in statistical mechanics and Gibbs conditioning principle}\label{section 3.1} The theory of large deviations has played an essential role in the equilibrium statistical mechanics. The sequence of probability distributions $\mu_n$ on the Polish space $\mathcal{S}$ are said to satisfy the \it{large deviation principle} (LDP) with the rate function $I$ provided that for all Borel sets $A$,
\begin{align*}
-\inf_{x\in A^{\mathrm{o}}} I(x)\leq \liminf_{n\rightarrow \infty} \frac{1}{n}\log \mu_n(A^{\mathrm{o}}) \leq \limsup_{n\rightarrow \infty} \frac{1}{n}\log \mu_n(\bar{A})\leq -\inf_{x\in \bar{A}} I(x).
\end{align*} 
We say that \it{weak LDP} holds when the  upper bound
\begin{align*}
\limsup_{n\rightarrow \infty} \frac{1}{n}\log \mu_n(\bar{A})\leq -\inf_{x\in \bar{A}} I(x)
\end{align*}
holds only for compact sets $\bar{A}$. We require the rate function $I:\mathcal{S}\rightarrow [0,\infty]$  to be lower semicontinuous. $I$ is said to be a \it{good} rate function if the set $\{x\in \mathcal{S}|I(x)\leq c\}$ is compact for any $c\in \R$.

Let us  consider the classical lattice system on $\mathbb{Z}$. Its configuration space is denoted by $\Omega=\R^{\mathbb{Z}}$, equipped with the product topology and the corresponding Borel field $\mathcal{B}$. In the absence of interactions between the  particles, a thermodynamic behavior of the empirical distributions can be described by the Sanov's large deviation theorem. It states that under the independent and identically distributed (i.i.d.)
 law $P = \lambda^{\mathbb{Z}}$ for some probability measure $\lambda$ on $\R$, the sequence of empirical distributions  $P(\frac{1}{n}( \delta_{X_1}+\cdots+\delta_{X_n})\in \cdot)$ satisfies the LDP with the rate function  given by \it{relative entropy}:
\begin{align*}
H(\mu|\lambda):= &\begin{cases}
\int \frac{d\mu}{d\lambda}\log \frac{d\mu }{d\lambda}d\lambda \quad &\mu \ll \lambda, \\
0 \quad &\text{otherwise}.
\end{cases}
\end{align*}
 This LDP result has been extended to the general Gibbs measures on $d$-dimensional underlying space $\mathbb{Z}^d$ in the presence of bounded and translation-invariant interaction potentials. We refer to \cite{cd,dv,fo,g1,g2,o} for the details and \cite{g1,rs} for a monograph on Gibbs measures.

Once we have a large deviation principle for the sequence of probability distributions, we are able to study  asymptotic behaviors of the conditional distributions. This can be rigorously stated as follows, which is called the  \it{Gibbs conditioning principle}:

\begin{theorem}\cite[Theorem 7.1]{l2}\label{theorem 2.1}
Let $\mathbb{P}_n$ be probability distributions on the Polish space $\mathcal{S}$ satisfying the large deviation principle with a good rate function $I$. Suppose that $F$ and $F_\epsilon$ ($\epsilon>0$) are  closed sets in $\mathcal{S}$ such that \\
(i) $I(F):=\inf_{x\in F} I(x)<\infty$, \\
(ii) $\mathbb{P}_n(F_\epsilon)>0$ for all $n$ and $\epsilon>0$, \\
(iii) $F=\cap_{\epsilon>0} F_\epsilon$,  \\
(iv) $F\subset (F_\epsilon)^{\mathrm{o}}$ for all $\epsilon>0$. \\
Define $M_F$ be a collection of $x\in F$ that minimize $I$ over the set $F$. Then, for any open set $G$ containing $M_F$,
\begin{align*}
\limsup_{\epsilon \rightarrow 0} \limsup_{n\rightarrow \infty} \frac{1}{n}\log \mathbb{P}_n(G^c|F_\epsilon)<0.
\end{align*}
If in addition $M_F=\{x_0\}$ is a singleton, then
\begin{align*}
\lim_{\epsilon \rightarrow 0} \lim_{n\rightarrow \infty} \mathbb{P}_n(\cdot |F_\epsilon) = \delta_{x_0}.
\end{align*}
\end{theorem}

As an application of the Gibbs conditioning principle, one can deduce the  following  classical result in the equilibrium statistical mechanics, which is called the principle of \it{equivalence of ensembles}:
\begin{theorem}\cite[Chapter 5]{rs} \label{equivalence}
Let $\lambda\in \mathcal{M}_1(\mathcal{S})$ and $\phi:\mathcal{S}\rightarrow \R$ be a bounded continuous function. Let us define $a:=\lambda$-ess inf $\phi$ and $b:=\lambda$-ess sup $\phi$. For $\beta\in \R$,  denote $\mu_{\beta}$ by the probability measure on $\mathcal{S}$ of the form:
\begin{align*}
d\mu_{\beta} = \frac{1}{Z_\beta}e^{-\beta \phi}d\lambda.
\end{align*} 
Suppose that $\{X_k\}$'s are i.i.d. with marginal $\lambda$.
Then, for $z\in (a,b)$, there exists a unique $\beta$ such that
\begin{align*}
\lim_{\delta\rightarrow 0^+}\lim_{n\rightarrow \infty} \mathbb{P}\bigg(X_1\in \cdot \ \bigg \vert \left|\frac{\phi(X_1)+\cdots+\phi(X_n)}{n}-z\right|\leq \delta\bigg) = \mu_{\beta}.
\end{align*}
Here, the inverse temperature $\beta$ is chosen to satisfy 
\begin{align*}
\int_\mathcal{S} \phi d\mu_\beta = z.
\end{align*}
\end{theorem}
It is not hard to check that similar result holds under the several constraints \eqref{several} with bounded and continuous observables $\phi_i$'s.
We refer to  \cite{dsz} for the generalized version of Theorem \ref{equivalence}, where the constraint is given by the bounded continuous interacting potentials. See also \cite{g3} for the case when the constraint is given by  possibly unbounded interactions.

\subsection{Large deviations for the joint law of empirical distributions and several empirical means} \label{section 3.2} In this section, we obtain the large deviation results for the joint law of empirical distributions and several empirical means, which will play a crucial role in proving Theorem \ref{theorem 1.1}.

\begin{theorem} \label{theorem 2.2}
Under the reference measure $\mathbb{P}$, the sequence $(L_n, S^1_n,\cdots,S^k_n)$ in $\mathcal{M}_1(\R^+)\times (\R^+)^k$ satisfies the weak LDP with a rate function $J$ given by
\begin{align*}
J&(\mu,v_1,\cdots,v_k)\\
&=\begin{cases} H(\mu | \lambda) &\text{if} \ \int \phi_1d\mu=v_1,\cdots,\int \phi_{k-1}d\mu=v_{k-1},\int \phi_kd\mu\leq v_k, \\
\infty &\text{otherwise}. \end{cases}
\end{align*}
\end{theorem}
\begin{proof}
We follow the argument in \cite{kr}. We apply   \cite[Theorem 6.1.3]{dz} to obtain the weak LDP for the sequence $(L_n, S^1_n,\cdots, S^k_n)$. This sequence is the empirical mean of the i.i.d random variables $(\delta_{X_i},\phi_1(X_i),\cdots,\phi_k(X_i))$  taking values in $ \mathcal{M}_1(\R^+)\times (\R^+)^k$. Let us denote $\mathcal{X}:= \mathcal{M}(\R^+) \times \R^k$, which is equipped with the product topology of weak topology on the space of measures and the standard topology on $\R^k$, and similarly define $\mathcal{E}:=\mathcal{M}_1(\R^+)\times (\R^+)^k$. It is not hard to check that Assumption 6.1.2 in \cite{dz} is satisfied in this setting (see \cite[Lemma 3.2]{kr} for  explanations in the case of $k=1$).
Thus, applying \cite[Theorem 6.1.3]{dz}, one can conclude that under the reference measure $\mathbb{P}$, the sequence $(L_n, S^1_n,\cdots, S^k_n)$ satisfies the weak LDP with a rate function $J$  given by
\begin{align} \label{221}
J(&\mu,v_1,\cdots,v_k) \nonumber\\
&=\sup_{f\in C_b(\R^+),p_1,\cdots,p_k\in \R} \Big\{\int fd\mu+p_1v_1+\cdots+p_kv_k-\log \int e^{f+p_1 \phi_1+\cdots+p_k \phi_k}d\lambda \Big\}.
\end{align}
It is easy to check that for any function $f\in C_b(\R^+)$, $(p_1,\cdots,p_k)$ satisfies
\begin{align*}
\int e^{f+p_1 \phi_1+\cdots+p_k \phi_k} d\lambda < \infty
\end{align*}
if and only if $(p_1,\cdots,p_k)$ belongs to the set
\begin{multline} \label{D}
D:=\{p_k<0\} \cup \{p_k=0,p_{k-1}<0\} \cup \cdots \cup \{p_k=p_{k-1}=\cdots=p_3=0,p_2<0\} \\
\cup \{p_k=p_{k-1}=\cdots=p_3=p_2=0,p_1<1\}
\end{multline}
thanks to the Assumption \ref{assume0}. Thus, it suffices to take the supremum over the set $D$ in the expression \eqref{221}. For each $(p_1,\cdots,p_k)\in D$, let us define the auxiliary probability measure $\nu^{(p_1,\cdots,p_k)}$ on $(0,\infty)$ whose distribution is given by
\begin{align*}
\frac{1}{Z^{(p_1,\cdots,p_k)}}e^{p_1 \phi_1+\cdots+p_k \phi_k}d\lambda
\end{align*}
($Z^{(p_1,\cdots,p_k)}$ is a normalizing constant).
Using the variation formula for the relative entropy:
\begin{align*}
H(\mu|\nu) = \sup_{f\in C_b} \Big\{\int fd\mu-\log \int e^f d\nu\Big\},
\end{align*}
one can rewrite \eqref{221} as
\begin{align*}
J&(\mu,v_1,\cdots,v_k) \\
&=\sup_{(p_1,\cdots,p_k)\in D} \Big\{p_1v_1+\cdots+p_kv_k-\log Z^{(p_1,\cdots,p_k)}+\sup_{f\in C_b}\big(\int fd\mu-\log \int e^{f} d\nu^{(p_1,\cdots,p_k)}\big) \Big\} \\
&=\sup_{(p_1,\cdots,p_k)\in D} \Big\{p_1v_1+\cdots+p_kv_k-\log Z^{(p_1,\cdots,p_k)}+H(\mu | \nu^{(p_1,\cdots,p_k)})\Big\} \\
&= \sup_{(p_1,\cdots,p_k)\in D} \Big\{p_1v_1+\cdots+p_kv_k+ H(\mu | \lambda)- \int (p_1 \phi_1+\cdots+p_k \phi_k)d\mu \Big\}.
\end{align*}
If we define the set $\mathcal{T}\subset\mathcal{M}_1(\R^+)$ by
\begin{align*}
\mathcal{T}:= \Big\{\mu\in \mathcal{M}_1(\R^+) \Big \vert \int \phi_1d\mu=v_1,\cdots,\int \phi_{k-1}d\mu=v_{k-1},\int \phi_kd\mu\leq v_k\Big\},
\end{align*}
then one can easily check that $J(\mu,v_1,\cdots,v_k)=H(\mu|\lambda)$ when $\mu\in \mathcal{T}$ and $\infty$ otherwise.
\end{proof}

\begin{remark}  \label{remark 2.3}
We present several remarks regarding Theorem \ref{theorem 2.2}.

1. If each function $\phi_i$' is bounded and continuous, then it is obvious that the sequence $(L_n, S^1_n,\cdots,S^k_n)$  satisfies the (full) LDP with a rate function $J^{\text{bounded}}$ defined by
\begin{align*}
J^{\text{bounded}}&(\mu,v_1,\cdots,v_k)\\
&=\begin{cases} H(\mu | \lambda) &\text{if} \ \int \phi_1d\mu=v_1,\cdots, \int \phi_kd\mu= v_k, \\
\infty &\text{otherwise}. \end{cases}
\end{align*}
Theorem \ref{theorem 2.2} implies that when $\phi_i$'s are unbounded functions satisfying Assumption \ref{assume0}, the rate function $J(\mu,v_1,\cdots,v_k)$ may be finite even when $\int  \phi_k d\mu \neq v_k$ since the weak topology induced on the space of probability measures is not strong enough to capture the behavior near the infinity. Note that $J(\mu,v_1,\cdots,v_k)=\infty$ if $\int \phi_i d\mu \neq v_i$ for some  $1\leq i\leq k-1$ since $\phi_k$ controls other functions $\phi_1,\cdots,\phi_{k-1}$.

2. When we consider the pair of  empirical distributions  and a single empirical mean, the large deviation result  Theorem \ref{theorem 2.2} reads as follows (see \cite[Lemma 3.3]{kr} in the case of $\phi(x) = x^p$ under the generalized Gaussian distribution): under the reference measure $(\frac{1}{Z}e^{-\phi}dx)^{\otimes \mathbb{N}}$, the sequence $(L_n,\frac{\phi(X_1)+\cdots+\phi(X_n)}{n})$ satisfies the (full) LDP with a good rate function
\begin{align*}
J(\mu,v)
&=\begin{cases} H(\mu | \lambda)+v-\int \phi d\mu &\text{if} \ \int \phi d\mu\leq v,\\
\infty &\text{otherwise}. \end{cases}
\end{align*}

3. For any fixed positive integer $j$, let us define
\begin{align*}
S^i_{n-j} := \frac{\phi_i(X_{j+1})+\cdots+\phi_i(X_n)}{n-j}
\end{align*}
for each $1\leq i\leq k$. Then, under the reference measure $\mathbb{P}$, the sequence $(L_n,S^1_{n-j},\cdots,S^k_{n-j})$ satisfies the weak LDP with the same rate function $J$ defined in Theorem \ref{theorem 2.2}. Indeed, two sequences  $(L_n,S^1_n,\cdots,S^k_n)$ and $(L_n,S^1_{n-j},\cdots,S^k_{n-j})$ are exponentially equivalent since for any realization,
\begin{align*}
\limsup_{n \rightarrow \infty} d\Big(\frac{1}{n}(\delta_{X_1}+\cdots+\delta_{X_n}), \frac{1}{n-j}(\delta_{X_{j+1}}+\cdots+\delta_{X_n})\Big)\leq  \limsup_{n \rightarrow \infty} \frac{2j}{n} = 0
\end{align*}
($d$ denotes the metric \eqref{metric}).

4. For the probability measure $\mu\ll dx$ satisfying the condition
\begin{align} \label{230}
\int \phi_1d\mu=v_1,\cdots,\int \phi_{k-1}d\mu=v_{k-1},\int \phi_kd\mu\leq v_k,
\end{align}  the rate function $J$ in Theorem \ref{theorem 2.2} can be written in terms of the differential entropy $h(\cdot)$: for some constant $C$,
\begin{align} \label{2300}
H(\mu | \lambda ) &= \int \log(\frac{d\mu }{d\lambda})d\mu = \int \log(\frac{d\mu }{dx})d\mu +\int \log(\frac{dx}{d\lambda})d\mu  \nonumber \\
&= -h(\mu) + \int \phi_1 d\mu + C = -h(\mu) + v_1 + C.
\end{align}
In general, if the reference measure $\lambda$ on $(0,\infty)$ is given by
\begin{align*}
\lambda = \frac{1}{Z}e^{p_1\phi_1+\cdots+p_k\phi_k}dx
\end{align*}
for some $(p_1,\cdots,p_k)$ for which the normalizing constant $Z$ is finite,
then the sequence $(L_n,S^1_n,\cdots,S^k_n)$ under $\lambda^{\otimes \mathbb{N}}$ satisfies the weak LDP with the rate function $J$  given by
\begin{align*}
J&(\mu,v_1,\cdots,v_k)\\
&=\begin{cases} H(\mu | \lambda) -p_k(v_k-\int \phi_k d\mu) &\text{if} \ \int \phi_1d\mu=v_1,\cdots,\int \phi_{k-1}d\mu=v_{k-1},\int \phi_kd\mu\leq v_k, \\
\infty &\text{otherwise}. \end{cases}
\end{align*}
For a probability measure $\mu \ll dx$ satisfying the condition \eqref{230}, the rate function $J$ can be written as
\begin{align*}
H(\mu | \lambda) -p_k(v_k-\int \phi_k d\mu) &= -h(\mu) -p_1\int\phi_1 d\mu - \cdots - p_k \int \phi_k d\mu -p_k(v_k-\int \phi_k d\mu) \\
&=  -h(\mu) -p_1v_1-\cdots- p_kv_k + C.
\end{align*}
\end{remark}

We need the following lemma to ensure the existence and uniqueness of a minimizer of the relative entropy $H(\cdot | \lambda)$ over the set \eqref{11}.
\begin{lemma}  \label{lemma 2.4}
The following set is closed, compact, and convex:
\begin{align*}
\mathcal{T}:=\Big\{\mu\in \mathcal{M}_1(\R^+)  \Big \vert \int \phi_1d\mu=v_1,\cdots,\int \phi_{k-1}d\mu=v_{k-1},\int \phi_kd\mu\leq v_k\Big\}.
\end{align*}

\end{lemma}
\begin{proof}
Suppose that $\mu_n\in \mathcal{T}$ and $\mu_n\rightarrow \mu$. We first show the closedness of $\mathcal{T}$ by proving that $\mu\in \mathcal{T}$. According to the Portmanteau theorem, $\int \phi_k d\mu \leq v_k$ is obvious.  Fix any $1\leq i\leq k-1$ and let us show that $\int \phi_i d\mu = v_i$. Using Assumption \ref{assume0} and the fact that $\int \phi_k d\mu_n \leq v_k$, one can conclude that for any $\epsilon>0$, there exists $M>0$ such that for all $n$,
\begin{align*}
 \int \phi_i \1_{[M,\infty)} d\mu_n <\epsilon.
\end{align*}
 This implies that $\int \phi_i \1_{(0,M)} d\mu_n > v_i-\epsilon$. Since $\mu_n\rightarrow \mu$ and $\phi_i \1_{(0,M)}  \in C_b(\R^+)$, we have
\begin{align*}
\lim_n \int \phi_i \1_{(0,M)}  d\mu_n  =  \int  \phi_i  \1_{(0,M)}  d\mu.
\end{align*}  Therefore, we have $\int \phi_i \1_{(0,M)}  d\mu \geq v_i-\epsilon$, and since $\epsilon$ is arbitrary, we obtain $\int \phi_i d\mu \geq v_i$. On the other hand, thanks to the Portmanteau theorem,  $\int \phi_i d\mu \leq v_i$. Thus, $\int \phi_i d\mu = v_i$, which  concludes the closedness of $\mathcal{T}$.

Compactness of $\mathcal{T}$ immediately follows from the Prokhorov's theorem and Assumption \ref{assume0}. Convexity of $\mathcal{T}$ is also obvious.
\end{proof}

According to the Sanov's theorem, the sequence of empirical distributions $L_n$  satisfy the LDP with a rate function $H(\cdot | \lambda)$. Also, due to the generalized version of Cram\'er's theorem (see  \cite[Theorem 6.1.3]{dz}), the sequence $(S^1_n,\cdots,S^k_n)$  satisfies the weak LDP with a rate function $I(v_1,\cdots,v_k)$ which is the Legendre transform of the logarithmic moment generating function:
\begin{align*}
H(p_1,\cdots,p_k) = \log \int e^{p_1\phi_1+\cdots+p_k\phi_k} d\lambda.
\end{align*}
 Since a map $\mu \rightarrow \int \phi_i d\mu$ may not be continuous, the rate function $I$ cannot be directly obtained from the Sanov's theorem as a simple application of the standard contraction principle. However, applying Theorem \ref{theorem 2.2}, one can obtain the non-continuous version of the contraction principle. It reveals the relation between two rate functions $H(\cdot | \lambda)$ and $I$.

\begin{proposition} \label{prop 2.4} Under the reference measure $\mathbb{P}$, the sequence $(S^1_n, \cdots, S^k_n)$ in $(\R^+)^k$ satisfies the weak LDP with a rate function $I(v_1,\cdots,v_k)$ given by
\begin{align} \label{240}
I(v_1,\cdots,v_k)&=\inf_{\mu\in \mathcal{M}_1(\R^+)} \Big\{H(\mu | \lambda)\ \Big \vert \int \phi_1d\mu=v_1,\cdots,\int \phi_{k-1}d\mu=v_{k-1},\int \phi_kd\mu\leq v_k \Big\}.
\end{align}
Also, $I(v_1,\cdots,v_k)$ is the Legendre transform of $H(p_1,\cdots,p_k)$ defined in \eqref{ham}.
\end{proposition}
\begin{proof}
Let us apply the contraction principle to the projection 
$\pi:(L_n,S^1_n,\cdots,S^k_n)\rightarrow (S^1_n,\cdots,S^k_n)$. Since $J$ is not necessarily a good rate function, in order that contraction principle works, we need to check that for $I(v_1,\cdots,v_k)$ defined in \eqref{240},
 \begin{align} \label{241}
\{(v_1,\cdots,v_k) | I(v_1,\cdots,v_k)\leq c\} = \pi(\{(\mu,v_1,\cdots,v_k) |J(\mu,v_1,\cdots,v_k) \leq c \})
\end{align}
holds, and that this set is a closed set (see the proof of   \cite[Theorem 4.2.1]{dz}). Note that since
\begin{align*}
\mathcal{T}:=\Big\{\mu\in \mathcal{M}_1(\R^+) \Big \vert \int \phi_1d\mu=v_1,\cdots,\int \phi_{k-1}d\mu=v_{k-1},\int \phi_kd\mu\leq v_k\Big\}
\end{align*}  is a closed  set according to Lemma \ref{lemma 2.4}, the infimum of $H(\cdot | \lambda)$ is attained over $\mathcal{T}$ when $I(v_1,\cdots,v_k)<\infty$.  This implies that the equality in \eqref{241} holds. Also, since the sub-level set $\{H(\cdot | \lambda)\leq c\}$ is compact with respect to the weak topology, under the projection $\pi$, the image of $\{(\mu,v_1,\cdots,v_k) | J(\mu,v_1,\cdots,v_k) \leq c \}$  is closed. 
Therefore, the contraction principle is applicable, and the sequence $(S^1_n,\cdots,S^k_n)$ satisfies the weak LDP with a rate function 
\begin{align*}
I(v_1,\cdots,v_k):= \inf_{\mu\in \mathcal{M}_1(\R^+)} J(\mu,v_1,\cdots,v_k),
\end{align*}
which immediately implies \eqref{240}. Also, due to the uniqueness property of the rate function, the second part of proposition is obvious.
\end{proof}
\begin{remark} \label{remark 2.5}
From the definition of the admissible set, for $(a_1,\cdots,a_k)\in \mathcal{A}$,  \begin{align*}
\Big\{ \mu\in \mathcal{M}_1(\R^+) \Big \vert h(\mu)\neq -\infty, \int \phi_1d\mu=a_1,\cdots,\int \phi_{k-1}d\mu=a_{k-1},\int \phi_kd\mu\leq a_k\Big\} 
\end{align*} is a non-empty set. Thus, according to Proposition \ref{prop 2.4},  whenever $(a_1,\cdots,a_k)\in \mathcal{A}$, $I(a_1,\cdots,a_k)<\infty$ (see the identity \eqref{2300}). This implies that for $(a_1,\cdots,a_k)\in \mathcal{A}$, the microcanonical distribution $\mathbb{P}((X_1,\cdots,X_n)\in \cdot | C^\delta_n)$ is well-defined since for each $\delta>0$,
\begin{align*}
\liminf_{n\rightarrow \infty} \frac{1}{n}\log \mathbb{P}(C^\delta_n) \geq -I(a_1,\cdots,a_k) > -\infty.
\end{align*}
On the other hand, when $a_k < g_1(a_1,\cdots,a_{k-1})$, it is obvious that    $I(a_1,\cdots,a_k)=\infty$.
\end{remark}
Now, let us define $\lambda^* = \lambda^*(a_1,\cdots,a_k)$ to be a unique minimizer of the relative entropy $H(\cdot|\lambda)$ over the set
\begin{align*}
\Big\{\mu\in \mathcal{M}_1(\R^+) \Big \vert \int \phi_1d\mu=a_1,\cdots,\int \phi_{k-1}d\mu=a_{k-1},\int \phi_kd\mu\leq a_k\Big\}.
\end{align*} 
The existence and uniqueness of a minimizer follows from Lemma \ref{lemma 2.4} and the lower semicontinuity, compact sublevel sets,  strict convexity properties of the relative entropy functional $H(\cdot | \lambda)$. Note that $\lambda^*$ is also a unique maximizer of the differential entropy $h(\cdot)$ due to the identity \eqref{2300}.

\subsection{Proof of Theorem \ref{theorem 1.1}} \label{section 3.3} In this section, we conclude the proof of Theorem \ref{theorem 1.1}. As an application of the Gibbs conditioning principle, combined with the large deviation result for the sequence $(L_n,S^1_n,\cdots,S^k_n)$ obtained in  Theorem \ref{theorem 2.2}, one can prove the following result:
\begin{lemma} \label{prop 2.6}
For any open set $G$ containing $\lambda^*$,
\begin{align*}
\limsup_{\delta \rightarrow 0} \limsup_{n\rightarrow \infty} \frac{1}{n}\log \mathbb{P}(L_n \notin G | C_n^\delta)<0.
\end{align*}
\end{lemma}
\begin{proof}
For each $0<\delta<\min \{a_1,\cdots,a_k\}$, define closed sets $F,F_\delta\subset \mathcal{M}_1(\R^+) \times (\R^+)^k$ by 
\begin{align*}
F=\{(L_n,S^1_n,\cdots,S^k_n)|S^1_n=a_1, \cdots,S^k_n=a_k\},
\end{align*}
\begin{align*}
F_\delta=\{(L_n,S^1_n,\cdots,S^k_n)|S^k_n\in [a_1-\delta,a_1+\delta],\cdots, S^k_n\in[a_k-\delta,a_k+\delta]\}.
\end{align*}
It is obvious that $F=\cap_{\delta>0} F_\delta$ and $F\subset  (F_\delta)^{\mathrm{o}}$. Since the infimum of $J(\mu,v_1,\cdots,v_k)$ over the constraint $v_1=a_1,\cdots,v_k=a_k$ is attained at $(\mu,v_1,\cdots,v_k)=(\lambda^*,a_1,\cdots,a_k)$, and $G\times (\R^+)^k $ is an open neighborhood of $(\lambda^*,a_1,\cdots,a_k)$,
according to the Gibbs conditioning principle Theorem \ref{theorem 2.1},
\begin{align*}
\limsup_{\delta \rightarrow 0} \limsup_{n\rightarrow \infty} &\frac{1}{n}\log \mathbb{P}(L_n \notin G | C_n^\delta)\\
&=\limsup_{\delta \rightarrow 0} \limsup_{n\rightarrow \infty} \frac{1}{n}\log \mathbb{P}((L_n,S^1_n,\cdots,S^k_n) \in G^c\times \R^{k} | C_n^\delta)<0.
\end{align*}
Note that even though the rate function $J$ is not necessarily a good rate function, Theorem \ref{theorem 2.1} is applicable since each $F_\delta$ is compact in $(\R^+)^k$-variable and $P(L_n\in \cdot)$ is exponentially tight.
\end{proof}

As a corollary of the previous lemma, one can finish the proof of Theorem \ref{theorem 1.1}.
\begin{proof}[Proof of Theorem \ref{theorem 1.1} ]
Recall that a unique maximizer of the differential entropy $h(\cdot)$ over the set \eqref{11} coincides with a unique minimizer of the relative entropy $H(\cdot | \lambda)$ over the same set \eqref{11} (see the identity \eqref{2300}). As a consequence of Lemma \ref{prop 2.6}, we have
\begin{align*}
\lim_{\delta \rightarrow 0} \lim_{n\rightarrow \infty} \mathbb{P}(L_n\in \cdot | C_n^\delta) = \delta_{\lambda^*}.
\end{align*}
According to \cite[Proposition 2.2]{sz}, this implies that  for any fixed positive integer $j$,
\begin{align*}
\lim_{\delta \rightarrow 0} \lim_{n\rightarrow \infty} \mathbb{P}((X_1,\cdots,X_j)\in \cdot | C^\delta_n) \to (\lambda^*)^{\otimes j}.
\end{align*}
\end{proof}

\begin{remark} \label{uniform}
Note that according to \eqref{321},  Lemma \ref{prop 2.6} also holds under the uniform distribution on the constraint $C^\delta_n$, which is of our main interest. Thus, the result in Theorem \ref{theorem 1.1} holds under the uniform distribution as well.
\end{remark}

\subsection{Characterization of the maximizer in Theorem \ref{theorem 1.1}} \label{section 3.4}
In this section, we characterize the (unique) maximizer of the differential entropy $h(\cdot)$ over the set \eqref{11}. Interestingly, it turns out that the maximizers have different forms in the case of $(a_1,\cdots,a_{k-1})\in \mathcal{S}_1$ and $(a_1,\cdots,a_{k-1})\in \mathcal{S}_2$. We first analyze the sets $\mathcal{S}_1$, $\mathcal{S}_2$ and the function $g_2$ defined in \eqref{def1} in a more detailed way.

\begin{remark} \label{remark 2.9}
For $(v_1,\cdots,v_{k-1})\in \mathcal{S}_1$, there exist  unique 
$p_1,\cdots,p_{k-1}$ satisfying \eqref{a1}. This can be verified  using the following facts: \\
(i) if $(v_1,\cdots,v_{k-1},z)\in \partial H(p_1,\cdots,p_{k-1},0)$ for some $p_1,\cdots,p_{k-1}$, then for all $z<w$, \\
 $(v_1,\cdots,v_{k-1},w)\in \partial H(p_1,\cdots,p_{k-1},0)$, \\
(ii) the rate function $I$ is differentiable on $\mathcal{A}$. \\
Since $H(p_1,\cdots,p_{k-1},p_k) = \infty$ for $p_k>0$, (i) follows from the definition of the subdifferential of convex functions. (ii) immediately follows from the essentially strictly convexity of $H$ and the fact that $\mathcal{A} \subset \text{dom}(I)$. In fact, the essentially strictly convexity of $H$ implies the essentially smoothness of $I$ (see   \cite[Theorem 26.3]{convex}). Since $\mathcal{A} \subset \text{dom}(I)$ and $\mathcal{A}$ is open, the essentially smoothness of $I$ implies that $I$ is differentiable on $\mathcal{A}$.

Suppose that there exist $(p_1,\cdots,p_{k-1})$ and $(p'_1,\cdots,p'_{k-1})$ satisfying \eqref{a1}. Using the fact (i), there exists $v_k$ such that $(v_1,\cdots,v_k)\in \mathcal{A}$ and
\begin{align*}
(v_1,\cdots,v_{k-1},v_k)\in \partial H(p_1,\cdots,p_{k-1},0),\quad (v_1,\cdots,v_{k-1},v_k)\in \partial H(p'_1,\cdots,p'_{k-1},0).
\end{align*} Since $H$ is convex and lower semicontinuous, using the duality of $H$ and $I$, we have
\begin{align*}
(p_1,\cdots,p_{k-1},0),(p'_1,\cdots,p'_{k-1},0)\in \partial I(v_1,\cdots,v_{k-1},v_k).
\end{align*}
Since $I$ is differentiable on $\mathcal{A}$, $(p_1,\cdots,p_{k-1})$ satisfying \eqref{a1} is unique. 
\end{remark}

The following lemma reveals useful properties of the sets $\mathcal{S}_1$, $\mathcal{S}_2$, and provides a formula for the function $g_2$.
\begin{lemma} \label{prop 2.9}
Suppose that $(v_1,\cdots,v_{k-1})\in \mathcal{S}_1$. Then, for $p_1,\cdots,p_{k-1}$ satisfying \eqref{a1},
\begin{align} \label{291}
v_i = \frac{1}{Z}\int \phi_i e^{p_1\phi_1+\cdots+p_{k-1}\phi_{k-1}}d\lambda
\end{align}
($Z$ is a normalizing constant $Z=\int  e^{p_1\phi_1+\cdots+p_{k-1}\phi_{k-1}}d\lambda$) for $1\leq i\leq k-1$ and
\begin{align} \label{292}
g_2(v_1,\cdots,v_{k-1}) = \frac{1}{Z}\int \phi_k e^{p_1\phi_1+\cdots+p_{k-1}\phi_{k-1}}d\lambda.
\end{align} 
Suppose that $(v_1,\cdots,v_{k-1})\in \mathcal{S}_2$. Then, for any $v_k$ such that $(v_1,\cdots,v_k)\in \mathcal{A}$, there exist $p_1,\cdots,p_k$ such that $p_k<0$ and
\begin{align} \label{294}
v_i = \frac{1}{Z}\int \phi_i e^{p_1\phi_1+\cdots+p_{k}\phi_{k}}d\lambda
\end{align}
($Z$ is a normalizing constant $Z=\int  e^{p_1\phi_1+\cdots+p_{k}\phi_{k}}d\lambda$) for $1\leq i\leq k$.
\end{lemma}

\begin{proof}
Let us consider the first case $(v_1,\cdots,v_{k-1})\in \mathcal{S}_1$. By the definition of the set $\mathcal{S}_1$, there exist $v_k$ and (unique) $p_1,\cdots,p_{k-1}$ satisfying
\begin{align} \label{295}
(v_1,\cdots,v_{k-1},v_k)\in \partial H(p_1,\cdots,p_{k-1},0).
\end{align}
This implies that for any $\epsilon>0$,
\begin{align*}
H(p_1,\cdots,p_{k-1},-\epsilon) - H(p_1,\cdots,p_{k-1},0) \geq -\epsilon v_k.
\end{align*}
Dividing this by $-\epsilon$ and then sending $\epsilon \rightarrow 0^+$, using Fatou's lemma, we obtain 
\begin{align} \label{296}
\frac{1}{Z}\int  \phi_k e^{p_1\phi_1+\cdots+p_{k-1}\phi_{k-1}}d\lambda \leq v_k.
\end{align}
Here, $Z$ is a normalizing constant $Z=\int  e^{p_1\phi_1+\cdots+p_{k-1}\phi_{k-1}}d\lambda$. This obviously implies that for any $1\leq i\leq k$,
\begin{align} \label{297}
w_i:=\frac{1}{Z}\int  \phi_i e^{p_1\phi_1+\cdots+p_{k-1}\phi_{k-1}}d\lambda < \infty.
\end{align}
Using \eqref{295} again, for each $1\leq i\leq k-1$ and for any $\epsilon>0$, $c\in \R$, we have
\begin{align*}
H(p_1,\cdots,p_i+\epsilon c,\cdots,p_{k-1},-\epsilon)-H(p_1,\cdots,p_i,\cdots,p_{k-1},0) \geq \epsilon c v_i - \epsilon v_k.
\end{align*}
This implies that
\begin{align} \label{ldct}
\lim_{\epsilon \rightarrow 0^+} \frac{1}{\epsilon}\Big[H(p_1,\cdots,p_i+\epsilon c,\cdots,p_{k-1},-\epsilon)-H(p_1,\cdots,p_i,\cdots,p_{k-1},0) \Big] \geq cv_i-v_k.
\end{align}
Using  dominated convergence theorem, let us check that left hand side of  \eqref{ldct} is equal to $cw_i-w_k$.
Indeed, if we denote $A,A_\epsilon$ ($\epsilon>0$) by
\begin{align*}
A:= e^{p_1\phi_1+\cdots+p_i\phi_i+\cdots+p_{k-1}\phi_{k-1}}, \ A_\epsilon:=  e^{p_1\phi_1+\cdots+(p_i+\epsilon c)\phi_i+\cdots +p_{k-1}\phi_{k-1} - \epsilon \phi_k},
\end{align*} then the left hand side of  \eqref{ldct} can be written as
\begin{align} \label{ldct3}
\lim_{\epsilon \rightarrow 0^+} \Big[\frac{\log \int A_\epsilon d\lambda- \log \int A d\lambda}{\int A_\epsilon d\lambda-\int A d\lambda}\cdot \frac{\int A_\epsilon d\lambda-\int A d\lambda}{\epsilon}\Big].
\end{align}
Note that
\begin{align*}
\lim_{\epsilon \rightarrow 0^+}\frac{ A_\epsilon - A }{\epsilon} =  \lim_{\epsilon \rightarrow 0^+} A\cdot \frac{e^{\epsilon(c\phi_i-\phi_k)}-1}{\epsilon} = c \phi_i A - \phi_k A.
\end{align*}
If we choose $M>0$ such that $x\geq M \Rightarrow c\phi_i(x)<\phi_k(x)$, then for $x\geq M$ and $\epsilon>0$,
\begin{align*}
\Big \vert A\cdot \frac{e^{\epsilon(c\phi_i-\phi_k)}-1}{\epsilon}\Big \vert \leq A (\phi_k-c\phi_i).
\end{align*}
Also, if we denote $N:= \sup_{0<x\leq M} |c\phi_i-\phi_k| <\infty$, then  for $x\in (0,M)$ and $0<\epsilon<1$,
\begin{align*}
\Big \vert A\cdot \frac{e^{\epsilon(c\phi_i-\phi_k)}-1}{\epsilon}\Big \vert\leq A (e^N-1).
\end{align*}
 Note that 
 $ A (\phi_k-c\phi_i) \in L^1(d\lambda)$ due to \eqref{297}, and $A\in L^1(d\lambda)$ since $(p_1,\cdots,p_{k-1},0)\in \text{dom}(H)$. Therefore, applying the dominated convergence theorem,
\begin{align} \label{ldct1}
\lim_{\epsilon \rightarrow 0^+}  \int \frac{ A_\epsilon - A }{\epsilon} d\lambda = \int (c\phi_i-\phi_k) e^{p_1\phi_1+\cdots+p_{k-1}\phi_{k-1}}d\lambda.
\end{align}
Also, since $\sup_{x\in (0,\infty)} (c\phi_i-\phi_k)<\infty$ and $A\in L^1(d\lambda)$, as an application of the dominated convergence theorem, one can deduce that $\lim_{\epsilon \rightarrow 0^+} \int A_\epsilon d\lambda = \int A d\lambda$. Thus,
\begin{align} \label{ldct2}
\lim_{\epsilon \rightarrow 0^+} \frac{\log \int A_\epsilon d\lambda- \log \int A d\lambda}{\int A_\epsilon d\lambda-\int A d\lambda} = \Big[\int e^{p_1\phi_1+\cdots+p_{k-1}\phi_{k-1}}d\lambda\Big]^{-1}.
\end{align}
Using \eqref{ldct3}, \eqref{ldct1} and \eqref{ldct2}, one can deduce that the left hand side of  \eqref{ldct} is equal to $cw_i-w_k$, and thus
 we have $cw_i-w_k \geq cv_i-v_k$. Since $c$ is arbitrary, we obtain  $w_i=v_i$, which implies \eqref{291}.
Also, using the convexity of $H$, it is easy to check that $(w_1,\cdots,w_{k-1},w_k)\in \partial H(p_1,\cdots,p_{k-1},0)$. Since any $v_k$ for which \eqref{295} holds satisfies \eqref{296},   we obtain \eqref{292}.

Finally, let us consider the case when $(v_1,\cdots,v_{k-1})\in \mathcal{S}_2$  and $(v_1,\cdots,v_{k-1},v_k)\in \mathcal{A}$.  Since $(v_1,\cdots,v_k)\in \mathcal{A}\subset \text{int}(\text{dom}(I))$ and $I$ is essentially smooth, $I$ is differentiable at $(v_1,\cdots,v_k)$. If we choose $(p_1,\cdots,p_k)\in \partial I(v_1,\cdots,v_k)$, then by the Legendre duality, we have $(v_1,\cdots,v_k)\in \partial H(p_1,\cdots,p_k)$. Since $(v_1,\cdots,v_{k-1})\in \mathcal{S}_2$, $p_k\neq  0$. This in turn implies that $p_k<0$ since $(p_1,\cdots,p_k)\in \text{dom}(\partial H) \subset \text{dom} (H) = D$. Thus, $H$ is differentiable at $(p_1,\cdots,p_k)$, and we immediately obtain \eqref{294}.
\end{proof}
Using Lemma \ref{prop 2.9}, one can characterize a (unique) maximizer of the differential entropy $h(\cdot)$ over the set \eqref{11}:

\begin{proposition} \label{theorem 2.10}
Assume that $\lambda^*$ is a unique maximizer of $h(\cdot)$ over the set \eqref{11}. In the case of $(a_1,\cdots,a_{k-1})\in \mathcal{S}_1$ and $a_k \geq g_2(a_1,\cdots,a_{k-1})$, 
\begin{align} \label{21000}
\lambda^*=\frac{1}{Z}e^{p_1\phi_1+\cdots+p_{k-1}\phi_{k-1}} dx
\end{align}
for $p_1,\cdots,p_{k-1}$ satisfying $\int \phi_i d\lambda^* = a_i$ for $1\leq i\leq k-1$.

On the other hand, either in the case of \\
(i) $(a_1,\cdots,a_{k-1})\in \mathcal{S}_2$ or \\
(ii) $(a_1,\cdots,a_{k-1})\in \mathcal{S}_1$ and $a_k < g_2(a_1,\cdots,a_{k-1})$,\\
\begin{align} \label{21001}
\lambda^* = \frac{1}{Z}e^{p_1\phi_1+\cdots+p_k\phi_k} dx
\end{align}
for $p_1,\cdots,p_k$ satisfying  $p_k<0$ and $\int \phi_i d\lambda^* = a_i$ for $1\leq i\leq k$. In all cases, $Z$ denotes the normalizing constant.
\end{proposition}
\begin{proof}
Let us first consider the case $(a_1,\cdots,a_{k-1})\in \mathcal{S}_1$ and $a_k \geq g_2(a_1,\cdots,a_{k-1})$. According to  Lemma \ref{prop 2.9}, there exists a probability measure $\nu$ of the form \eqref{21000} satisfying $\int \phi_i d\nu = a_i$ for $1\leq i\leq k-1$ and $\int \phi_k d\nu =  g_2(a_1,\cdots,a_{k-1}) \leq a_k$ (recall that $d\lambda$ is given by \eqref{basic}).   It is easy to check that $\nu$ is the maximizer of $h(\cdot)$ over the set \eqref{11}. In fact, for any probability measure $\mu \ll dx$,
\begin{align*}
-h(\mu) &= H(\mu | \nu) + p_1\int \phi_1 d\mu + \cdots + p_{k-1} \int \phi_{k-1}d\mu + C \\ 
&\geq p_1a_1+\cdots+p_{k-1}a_{k-1} + C,
\end{align*}
and the equality is attained if and only if $\mu = \nu$.

Let us now consider the other cases, (i) and (ii). In each case, we first show the existence of a probability measure $\nu$ of the form \eqref{21001} satisfying $\int \phi_i d\nu= a_i$ for $1\leq i\leq k$. In the case of (i), it is already proved in Lemma \ref{prop 2.9}, so we consider the case (ii). For $(p_1,\cdots,p_k)\in \partial I(a_1,\cdots,a_k)$,  we have $(a_1,\cdots,a_k)\in \partial H(p_1,\cdots,p_k)$  by the Legendre duality. Since $a_k<g_2(a_1,\cdots,a_{k-1})$, we have $p_k \neq 0$, which in turn implies $p_k<0$. This implies that $H$ is differentiable at $(p_1,\cdots,p_k)$, and for $1\leq i\leq k$,
\begin{align*}
a_i = \frac{1}{Z}\int \phi_i e^{p_1\phi_1+\cdots+p_k\phi_k} d\lambda.
\end{align*}
Now, as before, one can check that $\nu$ is the maximizer of $h(\cdot)$ over the set \eqref{11}. In fact, since $p_k<0$, for any probability measure $\mu \ll dx$,
\begin{align*}
-h(\mu) &= H(\mu | \nu) + p_1\int \phi_1 d\mu + \cdots + p_{k} \int \phi_{k}d\mu + C \\ 
&\geq p_1a_1+\cdots+p_{k}a_{k} + C,
\end{align*}
and the equality is attained if and only if $\mu = \nu$.
\end{proof}

\begin{proof}[Proof of Theorem \ref{theorem 1.0}]
 Theorem \ref{theorem 1.1} and Proposition \ref{theorem 2.10} immediately conclude the proof.
\end{proof}

\subsection{Structure of the rate function $I$} \label{section 3.5}
In this section, we establish  useful  properties of the rate function $I$. Recall that $I$ is the weak LDP rate function for the sequence $(S^1_n,\cdots,S^k_n)$ under the reference measure $\mathbb{P}$ (see Proposition \ref{prop 2.4}). It turns out that $I(v_1,\cdots,v_k)$ behaves differently when $(v_1,\cdots,v_{k-1})\in \mathcal{S}_1$ and $(v_1,\cdots,v_{k-1})\in \mathcal{S}_2$.

\begin{proposition} \label{prop 2.7}
For each $(v_1,\cdots,v_{k-1})\in (\R^+)^{k-1}$, the rate function $I(v_1,\cdots,v_{k-1},\cdot)$ is non-increasing. In the case of $(v_1,\cdots,v_{k-1})\in \mathcal{S}_1$,
\begin{align} \label{2100}
 I(v_1,\cdots,v_{k-1},z) > I(v_1,\cdots,v_{k-1},w)
\end{align}
for all $z,w\in \R^+$ satisfying $z<w\leq g_2(v_1,\cdots,v_{k-1})$ and $ (v_1,\cdots,v_{k-1},w)\in \mathcal{A}$. Also, $I(v_1,\cdots,v_{k-1},\cdot)$ is constant on the interval $[g_2(v_1,\cdots,v_{k-1}),\infty)$.

On the other hand, in the case of $(v_1,\cdots,v_{k-1})\in \mathcal{S}_2$, 
\begin{align} \label{2101}
I(v_1,\cdots,v_{k-1},z) > I(v_1,\cdots,v_{k-1},w)
\end{align}  
for all  $z,w\in \R^+$ satisfying $ z<w$ and $ (v_1,\cdots,v_{k-1},w)\in \mathcal{A}$.
\end{proposition}
\begin{proof} Proof consists of three steps.

Step 1. Non-increasing property on $(0,\infty)$: recall the variational formula:
\begin{align*}
I(v_1,\cdots,v_k) = \sup_{(p_1,\cdots,p_k)\in D} (p_1v_1+\cdots+p_kv_k-H(p_1,\cdots,p_k)),
\end{align*}
with the domain $D$ defined in \eqref{D}.
For each $(p_1,\cdots,p_k)\in D$, whenever $z<w$, 
\begin{align*}
p_1v_1+\cdots+p_{k-1}v_{k-1}&+p_kz - H(p_1,\cdots,p_k)\\
&\geq p_1v_1+\cdots+p_{k-1}v_{k-1}+p_kw  - H(p_1,\cdots,p_k).
\end{align*}
Thus,  $I(v_1,\cdots,v_{k-1},z) \geq I(v_1,\cdots,v_{k-1},w)$ when $z<w$. 

Step 2. Case $(v_1,\cdots,v_{k-1})\in \mathcal{S}_1$: if $z<g_1(v_1,\cdots,v_{k-1})$, then \eqref{2100} is obvious since $I(v_1,\cdots,v_{k-1},z)=\infty$ (see Remark \ref{remark 2.5}). Now, assume that for some $g_1(v_1,\cdots,v_{k-1})\leq z<w\leq g_2(v_1,\cdots,v_{k-1})$,
\begin{align*}
I(v_1,\cdots,v_{k-1},z)=I(v_1,\cdots,v_{k-1},w)<\infty.
\end{align*}
Since $I(v_1,\cdots,v_{k-1},\cdot)$ is non-increasing, $I(v_1,\cdots,v_{k-1},\cdot)$ is constant on the interval $[z,w]$. Thus, for any $y\in (z,w)$, the subgradient $(p_1,\cdots,p_k)$ of $I$ at $(v_1,\cdots,v_{k-1},y)$ should satisfy $p_k=0$. Since  $H$ and $I$ are conjugate to each other,
\begin{align*}
(v_1,\cdots,v_{k-1},y) \in \partial H(p_1,\cdots,p_{k-1},0).
\end{align*}
This contradicts the definition of $g_2(v_1,\cdots,v_{k-1})$ since $y<g_2(v_1,\cdots,v_{k-1})$. Thus, \eqref{2100} holds for $z,w$ satisfying $z<w\leq g_2(v_1,\cdots,v_{k-1})$ and $ (v_1,\cdots,v_{k-1},w)\in \mathcal{A}$.

Now, let us prove that $I(v_1,\cdots,v_{k-1},\cdot)$ is constant on the interval $[g_2(v_1,\cdots,v_{k-1}),\infty)$. Due to the definition of $g_2$ and the fact (i) in  Remark \ref{remark 2.9}, for arbitrary $\epsilon>0$, we have $(v_1,\cdots,v_{k-1},g_2(v_1,\cdots,v_{k-1})+\epsilon)\in \partial H(p_1,\cdots,p_{k-1},0)$. This implies that
\begin{multline} \label{21011}
I(v_1,\cdots,v_{k-1},g_2(v_1,\cdots,v_{k-1})+\epsilon) +H(p_1,\cdots,p_{k-1},0) \\
= (v_1,\cdots,v_{k-1},g_2(v_1,\cdots,v_{k-1})+\epsilon) \cdot  (p_1,\cdots,p_{k-1},0).
\end{multline}
Therefore, for any $x>0$, using \eqref{21011},
\begin{align*}
I(v_1,\cdots,&v_{k-1},g_2(v_1,\cdots,v_{k-1})+x) \\
&\geq  (v_1,\cdots,v_{k-1},g_2(v_1,\cdots,v_{k-1})+x)\cdot  (p_1,\cdots,p_{k-1},0) -H(p_1,\cdots,p_{k-1},0)  \\
&= (v_1,\cdots,v_{k-1},g_2(v_1,\cdots,v_{k-1})+\epsilon)\cdot  (p_1,\cdots,p_{k-1},0) -H(p_1,\cdots,p_{k-1},0)  \\
&= I(v_1,\cdots,v_{k-1},g_2(v_1,\cdots,v_{k-1})+\epsilon).
\end{align*}
 Since $x>0$ is arbitrary and $I(v_1,\cdots,v_{k-1}, \cdot)$ is non-increasing, it follows from the above inequality that  $I(v_1,\cdots,v_{k-1}, \cdot)$ is constant on the interval $[g_2(v_1,\cdots,v_{k-1})+\epsilon,\infty)$. Since $\epsilon>0$ is arbitrary and $I$ is lower semicontinuous, $I(v_1,\cdots,v_{k-1}, \cdot)$ is constant  on the interval $[g_2(v_1,\cdots,v_{k-1}),\infty)$.

Step 3. Case $(v_1,\cdots,v_{k-1})\in \mathcal{S}_2$: if $z<g_1(v_1,\cdots,v_{k-1})$, then \eqref{2101} is obvious since $I(v_1,\cdots,v_{k-1},z)=\infty$. Let us assume that for some $g_1(v_1,\cdots,v_{k-1})\leq z<w$,
\begin{align*}
I(v_1,\cdots,v_{k-1},z)=I(v_1,\cdots,v_{k-1},w)<\infty.
\end{align*}
Then, for any $y\in (z,w)$, the subgradient $(p_1,\cdots,p_k)$ of $I$ at $(v_1,\cdots,v_{k-1},y)$ should satisfy $p_k=0$. Thus, by the Legendre duality, we have
\begin{align*}
(v_1,\cdots,v_{k-1},y) \in \partial H(p_1,\cdots,p_{k-1},0),
\end{align*}
and this contradicts the definition of $\mathcal{S}_2$. Since we already proved the non-increasing property of $I(v_1,\cdots,v_{k-1},\cdot)$, proof  is concluded.
\end{proof}

Proposition \ref{prop 2.7} will play a crucial role in analyzing the localization and delocalization phenomena of the microcanonical ensembles in  Section \ref{section 4}.

\section{Localization and delocalization of microcanonical ensembles} \label{section 4}
 When the microcanonical ensemble is given by a single constraint, localization phenomenon does not happen in general (see Section \ref{section 5.1} and Proposition \ref{prop 5.2} for details).  However, when the microcanonical ensemble is given by multiple constraints, complicated localization  behaviors can happen, as  explained in Section \ref{section 2}. 
In this section, we systematically study the  localization and delocalization phenomena of such ensembles using the theory of  large deviations.

\subsection{Large deviations for the joint law of empirical distributions and the maximum component} \label{section 4.1}
Let us define the maximum component $M_n$ by
\begin{align*}
M_n:=\frac{\max_{1\leq i\leq n} \phi_k(X_i)}{n}.
\end{align*}
The key ingredient that reveals the localization behavior is the large deviation result for the maximum component $M_n$. In order to capture  the finer behavior of the microcanonical ensembles, we obtain a large deviation result for the sequence of the joint law $(L_n,M_n)$:

\begin{theorem} \label{theorem 3.1} 
For any Borel set $A$ in $ \mathcal{M}_1(\R^+) \times \R^+$, 
\begin{multline*}
 -\inf_{(\mu,z)\in A^{\mathrm{o}}} J^{max}_1(\mu, z)\leq  \liminf_{\delta \rightarrow 0} \liminf_{n\rightarrow \infty} \frac{1}{n}\log \mathbb{P}((L_n, M_n)\in A^{\mathrm{o}}|C_n^\delta) \\ \leq  \limsup_{\delta \rightarrow 0} \limsup_{n\rightarrow \infty} \frac{1}{n}\log \mathbb{P}((L_n, M_n)\in \bar{A}|C_n^\delta) \leq -\inf_{(\mu,z)\in \bar{A}} J^{max}_1(\mu, z),
\end{multline*}
with the rate function $J^{max}_1$ given by
\begin{align*}
J^{max}_1(\mu, z):=J(\mu, a_1,\cdots,a_{k-1},a_k-z)-I(a_1,\cdots,a_k).
\end{align*}
\end{theorem}
\begin{proof}
Throughout this proof, we use the notations 
\begin{align*}
S^i_{n-j} := \frac{\phi_i(X_{j+1})+\cdots+\phi_i(X_n)}{n-j}, \ L_{n-j}:= \frac{1}{n-j}(\delta_{X_{j+1}}+\cdots+\delta_{X_n}).
\end{align*}
for any fixed index $j$ and $1\leq i\leq k$. Also, for $r>0$ and $\mu\in \mathcal{M}_1(\R^+)$, define $B(\mu,r)$ and $\bar{B}(\mu,r)$ by
\begin{align*}
B(\mu,r):= \{\nu\in \mathcal{M}_1(\R^+) | d(\nu,\mu) < r\}, \quad \bar{B}(\mu,r):= \{\nu\in \mathcal{M}_1(\R^+) | d(\nu,\mu) \leq r\}.
\end{align*}
Recall that $d$ is a metric defined in \eqref{metric} that induces the weak convergence of probability measures. 

Step 1. Upper bound large deviations: it is obvious that
\begin{align} \label{310}
\mathbb{P}((L_n, M_n) \in \bar{A} | C_n^\delta) \leq n \mathbb{P}\Big(\big(L_n,\frac{\phi_k(X_1)}{n}\big) \in \bar{A}| C_n^\delta\Big).
\end{align}
According to the LDP for the sequence $(S^1_n,\cdots,S^k_n)$, for each $\delta>0$, we have
\begin{align} \label{311}
\liminf_{n\rightarrow \infty} \frac{1}{n}\log \mathbb{P}(C_n^\delta) \geq -\inf_{v_i\in (a_i-\delta,a_i+\delta)} I(v_1,\cdots,v_k) \geq -I(a_1,\cdots,a_k).
\end{align}
Let us define $A^\delta$ by a collection of $(\mu,y)\in \mathcal{M}_1(\R^+) \times \R^+$ for which there exists $x\in \R^+$ satisfying $(\mu,x)\in \bar{A}$ and $|y-(a_k-x)|<\delta$.
Then, using the condition (C3) in Assumption \ref{assume0},  for sufficiently large $n$,
 \begin{align*}
 \Big\{\big(L_n,\frac{\phi_k(X_1)}{n}\big)\in \bar{A}\Big\} \cap C_n^\delta \Rightarrow  B^\delta_n:= \cap_{i=1}^{k-1} \{|S^i_{n-1} - a_i|<2\delta \} \cap \{(L_n,S^k_{n-1}) \in A^\delta\}.
 \end{align*}
  According to the LDP result Theorem \ref{theorem 2.2} and Remark \ref{remark 2.3},
\begin{align} \label{312}
&\limsup_{n\rightarrow \infty} \frac{1}{n}\log \mathbb{P}((L_n,S^1_{n-1}\cdots,S^k_{n-1})\in B_n^\delta)  \nonumber \\
&\leq -\inf_{(\mu,v_k)\in A^\delta,  v_1\in [a_1-2\delta,a_1+2\delta],\cdots,v_{k-1}\in [a_{k-1}-2\delta,a_{k-1}+2\delta]} J(\mu,v_1,\cdots,v_k)
\end{align}
Note that since the sequence $\{L_n\}$ under $\mathbb{P}$ is exponentially tight and $\prod_{i=1}^{k-1} [a_i-2\delta,a_i+2\delta] \times [0,a_k+\delta]$ is compact, the weak LDP result Theorem \ref{theorem 2.2} is applicable.
Sending $\delta \rightarrow 0$, using   \cite[Lemma 4.1.6]{dz},
\begin{align} \label{313}
\lim_{\delta \rightarrow 0}& \inf_{(\mu,v_k)\in A^\delta,  v_1\in [a_1-2\delta,a_1+2\delta],\cdots,v_{k-1}\in [a_{k-1}-2\delta,a_{k-1}+2\delta]} J(\mu,v_1,\cdots,v_k) \nonumber \\
& = \inf_{(\mu,v_k)\in \bar{A}}J(\mu,a_1,\cdots,a_{k-1},a_k-z).
\end{align}
Note that although $J$ is not necessarily a good rate function,  \cite[Lemma 4.1.6]{dz} is applicable since intervals $[a_i-2\delta,a_i+2\delta]$ and $[0,a_k+\delta]$ are compact and the relative entropy has compact sub-level sets. Therefore, using \eqref{311}, \eqref{312}, and \eqref{313},
\begin{align*}
 &\limsup_{\delta \rightarrow 0}\limsup_{n\rightarrow \infty}\frac{1}{n}\log \mathbb{P}\Big(\big(L_n,\frac{\phi_k(X_1)}{n}\big)\in \bar{A}|C_n^\delta\Big) \\
&\leq  \limsup_{\delta \rightarrow 0}\limsup_{n\rightarrow \infty} \frac{1}{n}\log \mathbb{P}\Big(\Big\{\big(L_n,\frac{\phi_k(X_1)}{n}\big)\in \bar{A}\Big\} \cap C_n^\delta\Big) - \liminf_{\delta \rightarrow 0}\liminf_{n\rightarrow \infty} \frac{1}{n}\log \mathbb{P}(C_n^\delta) \\
&\leq -\inf_{(\mu,z)\in \bar{A}} J^{max}_1(\mu,z).
\end{align*}
This and \eqref{310} conclude the proof of upper bound large deviation.

Step 2. Lower bound large deviations: it suffices to show that  for any $z,\epsilon>0$ and open set $U$ containing arbitrary $\mu\in \mathcal{M}_1(\R^+)$,
\begin{align} \label{lower ldp}
-J^{max}_1(\mu,z)\leq  \liminf_{\delta \rightarrow 0} \liminf_{n\rightarrow \infty} \frac{1}{n}\log \mathbb{P}((L_n,M_n)\in U\times (z-\epsilon,z+\epsilon)|C_n^\delta).
\end{align}
If $\int \phi_k d\mu > a_k-z$ or $\int \phi_i d\mu \neq a_i$ for some $1\leq i\leq k-1$, then \eqref{lower ldp} is obvious since $J^{max}_1(\mu,z)=\infty$. Thus, throughout the proof we assume that $\int \phi_k d\mu \leq a_k-z$ and $\int \phi_i d\mu = a_i$ for $1\leq i\leq k-1$. Since $\phi_k$ is bounded from below and continuous, according to the   Portmanteau theorem, there exists $r_0>0$ such that
\begin{align} \label{cond}
d(\nu,\mu)<r_0 \Rightarrow \int \phi_k d\nu > \int \phi_kd\mu - z.
\end{align} 
Take a positive integer $j\geq 2$ such that
\begin{align*}
a_k-jz<\int \phi_k d\mu \leq a_k-(j-1)z,
\end{align*}
and denote $0\leq w:=a_k-(j-1)z-\int \phi_kd\mu<z$. Also, define two events $E^1_{n,\delta}$ and $E^2_{n}$ by
\begin{gather*}
E^1_{n,\delta} :=\bigcap_{i=1}^{j-1} \Big\{\frac{\phi_k(X_i)}{n}\in (z-\frac{\delta}{4(j-1)},z+\frac{\delta}{4(j-1)})\Big\}\bigcap \Big\{ \frac{\phi_k(X_{j})}{n} \in (w-\frac{\delta}{4}, w+\frac{\delta}{4})\Big\}, \\
E^2_{n}:=\bigcap_{i=j+1}^n \Big\{\frac{\phi_k(X_i)}{n}<z\Big\}.
\end{gather*}
It is obvious that for sufficiently small $\delta>0$,
\begin{align*}
E^1_{n,\delta}\cap E^2_{n} \Rightarrow M_n\in (z-\epsilon,z+\epsilon).
\end{align*}  Therefore, for the open set $U = B(\mu,r)$ with $r<\frac{r_0}{2}$, for  sufficiently small $\delta>0$,
\begin{align} \label{314}
&\liminf_{n\rightarrow \infty} \frac{1}{n}\log \mathbb{P}((L_n,M_n)\in U\times (z-\epsilon,z+\epsilon)|C_n^\delta) \nonumber \\
 &\geq \liminf_{n\rightarrow \infty} \frac{1}{n}\log \mathbb{P}(\{L_n\in U\}\cap E^1_{n,\delta}\cap E^2_{n} | C_n^\delta) \nonumber \\
&\geq \liminf_{n\rightarrow \infty} \frac{1}{n}\log \mathbb{P}(\{L_n\in U\}\cap E^1_{n,\delta}\cap E^2_{n} \cap C_n^\delta) - \limsup_{n\rightarrow \infty} \frac{1}{n}\log \mathbb{P}(C_n^\delta).
\end{align}
According to the LDP for the sequence $(S^1_n,\cdots,S^k_n)$ and   \cite[Lemma 4.1.6]{dz},
\begin{align} \label{315}
&\limsup_{\delta \rightarrow 0} \limsup_{n\rightarrow \infty} \frac{1}{n}\log \mathbb{P}(C_n^\delta) \nonumber \\
&\leq \limsup_{\delta \rightarrow 0} \Big[-\inf_{v_1\in [a_1-\delta,a_1+\delta],\cdots,v_k\in [a_k-\delta,a_k+\delta]}I(v_1,\cdots,v_k)\Big] = -I(a_1,\cdots,a_k).
\end{align}
Also, it is obvious that 
\begin{align} \label{3100}
&\mathbb{P}(\{L_n\in U\}\cap E^1_{n,\delta}\cap E^2_{n}\cap C_n^\delta) \nonumber \\
&= \mathbb{P}(\{L_n\in U\}\cap E^1_{n,\delta}\cap  C_n^\delta) - \mathbb{P}(\{L_n\in U\}\cap E^1_{n,\delta}\cap (E^2_{n})^c\cap C_n^\delta) \nonumber \\
&\geq \mathbb{P}(\{L_n\in U\}\cap E^1_{n,\delta}\cap  C_n^\delta)
- (n-j)\mathbb{P}\Big(\{L_n\in U\}\cap E^1_{n,\delta}\cap \big\{\frac{\phi_k(X_{j+1})}{n}\geq z\big\}\cap C_n^\delta\Big).
\end{align}
Let us first estimate the following quantity:
\begin{align*}
\liminf_{\delta \rightarrow 0} \liminf_{n\rightarrow \infty} \frac{1}{n}\log \mathbb{P}(\{L_n\in B(\mu,r)\} \cap E^1_{n,\delta} \cap  C_n^\delta).
\end{align*} For sufficiently small $\delta>0$, one can take  open sets $D^\delta_n$ in $(\R^+)^k$ such that for  sufficiently large $n$,
\begin{align*}
\prod_{i=1}^{k-1} \Big(a_i-\frac{\delta}{2},a_i+\frac{\delta}{2}\Big)\times \Big(\int \phi_kd\mu-\frac{\delta}{2},\int \phi_kd\mu+\frac{\delta}{2}\Big)\subset  D_n^\delta,
\end{align*}
\begin{align*}
E^1_{n,\delta}\cap
\{ (S^1_{n-j},\cdots,S^{k-1}_{n-j},S^k_{n-j})\in D^\delta_n \} \Rightarrow  E^1_{n,\delta}\cap C^\delta_n,
\end{align*}
thanks to the condition (C3) in Assumption \ref{assume0}.
Therefore, according to the LDP result Theorem \ref{theorem 2.2}, for sufficiently small $\delta>0$,
\begin{align} \label{316}
&\liminf_{n\rightarrow \infty} \frac{1}{n}\log \mathbb{P}(\{L_n\in B(\mu,r)\} \cap E^1_{n,\delta} \cap  C_n^\delta)  \nonumber\\
&\geq \liminf_{n\rightarrow \infty} \frac{1}{n}\log \mathbb{P}(\{L_{n-j}\in B(\mu, \frac{r}{2}) \}  \cap E^1_{n,\delta}\cap \{ (S^1_{n-j},\cdots,S^k_{n-j})\in D^\delta_n \}) \nonumber \\
&\geq \liminf_{n\rightarrow \infty}\frac{1}{n}\log \mathbb{P}( E^1_{n,\delta})+\liminf_{n\rightarrow \infty} \frac{1}{n}\log \mathbb{P}((L_{n-j}, S^1_{n-j},\cdots,S^k_{n-j})\in B(\mu, \frac{r}{2}) \times D_n^\delta) \nonumber \\
&\geq -\inf_{\nu\in B(\mu, \frac{r}{2}),(v_1,\cdots,v_k)\in D_n^\delta} J(\nu,v_1,\cdots,v_k) \nonumber \\
&\geq -J(\mu,a_1,\cdots,a_{k-1},\int \phi_kd\mu) = -H(\mu|\lambda) = -J(\mu,a_1,\cdots,a_{k-1},a_k-z).
\end{align}
Note that in the fourth line, we used \eqref{301} in Lemma \ref{lemma 3.1}.

Now, let us show that
\begin{align} \label{31000}
\limsup_{\delta \rightarrow 0}\limsup_{n\rightarrow \infty} &\frac{1}{n}\log \mathbb{P}\Big(\{L_n\in B(\mu,r)\}\cap E^1_{n,\delta}\cap \big\{\frac{\phi_k(X_{j+1})}{n}\geq z\big\}\cap C_n^\delta\Big)=-\infty.
\end{align}
Note that under $E^1_{n,\delta}\cap \{\frac{\phi_k(X_{j+1})}{n}\geq z\}\cap C_n^\delta$,
\begin{align*}
a_k+\delta \geq \frac{\phi_k(X_1)+\cdots+\phi_k(X_{j+1})}{n} > (j-1)z+w-\frac{\delta}{2}+z = a_k-\int \phi_k d\mu + z-\frac{\delta}{2},
\end{align*} 
which implies that $\int \phi_k d\mu > z-\frac{3\delta}{2}$. Thus, if $\int \phi_k d\mu<z$, then $E^1_{n,\delta}\cap \{\frac{\phi_k(X_{j+1})}{n}\geq z\}\cap C_n^\delta$ is an empty set for sufficiently small $\delta>0$, which  implies \eqref{31000}. Thus,  from now on we assume that $\int \phi_k d\mu \geq z$. One can take closed sets $F^\delta_n$ in $(\R^+)^k$ such that for  sufficiently large $n$,
\begin{align*}
F^\delta_n \subset \prod_{i=1}^{k-1} [a_i-2\delta,a_i +2\delta]\times [0,\int \phi_k d\mu-z+2\delta],
\end{align*}
\begin{align*}
E^1_{n,\delta}\cap \big\{\frac{\phi_k(X_{j+1})}{n}\geq z\big\}\cap C_n^\delta \Rightarrow  (S^1_{n-j-1},\cdots,S^k_{n-j-1})\in F^\delta_n.
\end{align*}
 Applying  the LDP result Theorem \ref{theorem 2.2} and Remark \ref{remark 2.3}, we have
\begin{align} \label{317}
\limsup_{n\rightarrow \infty} &\frac{1}{n}\log \mathbb{P}\Big(\{L_n\in B(\mu,r)\}\cap E^1_{n,\delta}\cap \big\{\frac{\phi_k(X_{j+1})}{n}\geq z\big\}\cap C_n^\delta\Big) \nonumber \\
&\leq \limsup_{n\rightarrow \infty} \frac{1}{n}\log \mathbb{P}((L_n, S^1_{n-j-1},\cdots,S^k_{n-j-1})\in \bar{B}(\mu,r)\times F^\delta_n) \nonumber \\
&\leq -\inf_{\nu\in \bar{B}(\mu,r), (v_1,\cdots,v_k)\in F^\delta_n }J(\nu, v_1,\cdots,v_k).
\end{align}
Taking a limit $\delta \rightarrow 0$, using \cite[Lemma 4.1.6]{dz},
\begin{align*}
\lim_{\delta \rightarrow 0} \inf_{\nu\in\bar{B}(\mu,r), (v_1,\cdots,v_k)\in F^\delta_n}J(\nu, v_1,\cdots,v_k) = \inf_{\nu\in \bar{B}(\mu,r),v_k\in [0, \int \phi_kd\mu-z]}J(\nu, a_1,\cdots,a_{k-1},v_k).
\end{align*}
Since we chose $r_0$ satisfying \eqref{cond} and $r<\frac{r_0}{2}$,
\begin{align*}
\inf_{\nu\in \bar{B}(\mu,r),v_k\in [0, \int \phi_kd\mu-z]}J(\nu, a_1,\cdots,a_{k-1},v_k)  = -\infty.
\end{align*}
Therefore, sending $\delta \rightarrow 0$ in \eqref{317}, one can deduce \eqref{31000}. 
Applying \eqref{31000} and \eqref{316} to \eqref{3100}, we obtain
\begin{align} \label{318}
\liminf_{\delta \rightarrow 0} \liminf_{n\rightarrow \infty} \frac{1}{n}\log \mathbb{P}(\{L_n\in B(\mu,r)\}\cap E^1_{n,\delta}\cap E^2_{n,\delta} \cap C_n^\delta)  \geq  -J(\mu,a_1,\cdots,a_k-z).
\end{align}
Thus, using \eqref{314}, \eqref{315}, and \eqref{318}, we finally obtain \eqref{lower ldp} since
\begin{align*}
\liminf_{\delta \rightarrow 0} &\liminf_{n\rightarrow \infty} \frac{1}{n}\log \mathbb{P}((L_n,M_n)\in B(\mu,r)\times (z-\epsilon,z+\epsilon)|C_n^\delta) \\
&\geq -J(\mu,a_1,\cdots,a_{k-1},a_k-z)+I(a_1,\cdots,a_k)  \\
&= -J^{max}_1(\mu,z).
\end{align*}
\end{proof}

\begin{proof}[Proof of Theorem \ref{theorem 1.2}]
Recall that when $\mu \ll dx$, $H(\mu | \lambda) = -h(\mu) + \int \phi_1 d\mu + C$ for some constant $C$. Thus, using Proposition \ref{prop 2.4} and the rate function formula for $J$ in Theorem \ref{theorem 2.2}, one can conclude that
\begin{align*}
&J^{max}_1(\mu,z) \\
=&\begin{cases} -h(\mu) - K(a_1,\cdots,a_k) &\text{if}\ \int \phi_1d\mu=a_1, \cdots, \int \phi_{k-1}d\mu=a_{k-1}, \int \phi_kd\mu \leq a_k-z ,\\
\infty &\text{otherwise}, \end{cases}
\end{align*}
for 
\begin{align*}
K(a_1,\cdots,a_k) = \inf_{\mu\in \mathcal{M}_1(\R^+)} \Big\{-h(\mu)\ \Big \vert \int \phi_1d\mu=a_1,\cdots,\int \phi_{k-1}d\mu=a_{k-1},\int \phi_kd\mu\leq a_k \Big\}.
\end{align*}
This concludes the proof of Theorem \ref{theorem 1.2}.
\end{proof}

\subsection{Localization and delocalization} \label{section 4.2}
In this section, we study the localization and delocalization phenomena using the large deviation result Theorem \ref{theorem 1.2}. First, we prove Theorem \ref{theorem 1.3}, which is about the delocalization result.

\begin{proof}[Proof of Theorem \ref{theorem 1.3}]

First, let us consider the case when $(a_1,\cdots,a_{k-1})\in \mathcal{S}_1$ and $a_k>g_2(a_1,\cdots,a_{k-1})$. Applying the LDP result Theorem \ref{theorem 3.1} and Proposition \ref{prop 2.4},
\begin{align*} 
\limsup_{\delta \rightarrow 0} &\limsup_{n\rightarrow \infty} \frac{1}{n}\log \mathbb{P}(M_n \in [a_k-g_2(a_1,\cdots,a_{k-1})+\epsilon,a_k] | C^\delta_n) \\
&\leq -\inf_{z\in [a_k-g_2(a_1,\cdots,a_{k-1})+\epsilon,a_k] }I(a_1,\cdots,a_{k-1},a_k-z) + I(a_1,\cdots,a_k)  \\
&= -\inf_{w\in [0,g_2(a_1,\cdots,a_{k-1})-\epsilon] }I(a_1,\cdots,a_{k-1},w) + I(a_1,\cdots,a_k)<0.
\end{align*}
The last inequality follows from  Proposition \ref{prop 2.7}.

Now, suppose that (i) or (ii) holds. Applying the LDP result Theorem \ref{theorem 3.1} and Proposition \ref{prop 2.4} again, we have
\begin{align*}
\limsup_{\delta \rightarrow 0} &\limsup_{n\rightarrow \infty} \frac{1}{n}\log \mathbb{P}(M_n \in [\epsilon,a_k] | C^\delta_n) \\
&\leq -\inf_{z\in [\epsilon,a_k]}I(a_1,\cdots,a_{k-1},a_k-z) + I(a_1,\cdots,a_k)\\
& =  -\inf_{w\in [0,a_k-\epsilon]}I(a_1,\cdots,a_{k-1},w) + I(a_1,\cdots,a_k)< 0.
\end{align*}
The last inequality follows from  Proposition \ref{prop 2.7}.

\end{proof}

We have shown that when $(a_1,\cdots,a_k)$ satisfies (i) or (ii) in Theorem \ref{theorem 1.3}, localization does not happen. We now consider the case when $(a_1,\cdots,a_{k-1})\in \mathcal{S}_1$ and $a_k>g_2(a_1,\cdots,a_{k-1})$.  As  explained in Section \ref{section 2}, unlike the upper tail estimate \eqref{upper} for the maximum component $M_n$,  the lower tail estimate:
\begin{align} \label{330}
\limsup_{\delta \rightarrow 0} &\limsup_{n\rightarrow \infty} \frac{1}{n}\log \mathbb{P}(M_n \leq a_k-g_2(a_1,\cdots,a_{k-1})-\epsilon | C^\delta_n) < 0
\end{align}
does not hold. In fact, according to the large deviation result Theorem \ref{theorem 3.1} and Proposition \ref{prop 2.7}, we have
\begin{align*}
\liminf_{\delta \rightarrow 0} &\liminf_{n\rightarrow \infty} \frac{1}{n}\log \mathbb{P}(M_n < a_k-g_2(a_1,\cdots,a_{k-1})-\epsilon | C^\delta_n) \\ 
&\geq -\inf_{z\in [0,a_k-g_2(a_1,\cdots,a_{k-1})-\epsilon)} I(a_1,\cdots,a_{k-1},a_k-z) + I(a_1,\cdots,a_{k-1},a_k) = 0.
\end{align*} 
As mentioned in Section \ref{section 2}, unlike the upper tail estimate \eqref{upper}, the correct scaling factor in the lower tail estimate of type \eqref{330} highly depends on the structures of  functions $\phi_i$'s. We now prove Theorem \ref{theorem 1.4}, which is about the lower tail estimate and the localization result. Since the correct scaling factor grows slowly than $n$, the proof is completely different from the standard large deviation arguments we have used so far, and we partially adapt the  idea  in \cite{cha2}.

\begin{proof}[Proof of Theorem \ref{theorem 1.4}] 
We partially follow the argument in \cite{cha2}. Recall that $1\leq m\leq k-1$ is the largest index such that $p_m\neq 0$, and it is obvious that $p_m<0$.  Throughout the proof, we define $s:=a_k-g_2(a_1,\cdots,a_{k-1})$ and choose a sufficiently small $\theta>0$ such that $p_m+3\theta<0$. In order to alleviate the notation, we define $\gamma:=\gamma_m$. Choose two numbers $0<\alpha,\beta<1$ satisfying
\begin{align} \label{340}
\frac{1}{2}(1+\gamma+2\alpha) < \beta < 1.
\end{align}
 We first compute  the lower bound of
\begin{align*}
\liminf_{n\rightarrow \infty} \frac{1}{n^{\gamma}}\log \mathbb{Q}(C_n^{\delta}).
\end{align*}
It is obvious that
\begin{align*}
\bigcap_{i=1}^{k-1}\big\{|S^i_{n-1}-a_i|< \frac{\delta}{2}\big\}\bigcap \big\{|S^k_{n-1} - g_2(a_1,\cdots,a_{k-1})| < \frac{\delta}{2}\big\} \bigcap  \big\{|\frac{\phi_k(X_1)}{n}-s|<\frac{\delta}{2}\big\} \Rightarrow C_n^{\delta}.
\end{align*}
Since  $\int \phi_i d\nu = a_i$ for $1\leq i\leq k-1$ and $\int \phi_k d\nu = g_2(a_1,\cdots,a_{k-1})$ (see Lemma \ref{prop 2.9}),  according to the law of large numbers, 
\begin{align} \label{3400}
\lim_{n\rightarrow \infty} \mathbb{Q}\Big(\bigcap_{i=1}^{k-1}\big\{|S^i_{n-1}-a_i|< \frac{\delta}{2}\big\}\bigcap \big\{|S^k_{n-1} - g_2(a_1,\cdots,a_{k-1})| < \frac{\delta}{2}\big\}\Big)=1.
\end{align}
Thus, combining \eqref{302} in Lemma \ref{lemma 3.1} with \eqref{3400}, we obtain
\begin{align} \label{341}
\liminf_{\delta \rightarrow 0}\liminf_{n\rightarrow \infty} \frac{1}{n^{\gamma}}\log \mathbb{Q}(C_n^{\delta}) \geq p_ms^{\gamma}.
\end{align}
Now, let us compute the upper bound of 
\begin{align*}
\limsup_{\delta \rightarrow 0} \limsup_{n\rightarrow \infty} \frac{1}{n^{\gamma}}\log \mathbb{Q}(\{M_n<s-\epsilon\} \cap C_n^\delta).
\end{align*}
For $n\in \mathbb{N}$, let us choose $f(n)$ satisfying $\phi_k(f(n))=n^\alpha$, and define   $u_n:=\E^\nu\big[ \phi_k(X_i) \1_{X_i\leq f(n)}\big]$. 
Note that $f$ is increasing and $\lim_{n\rightarrow \infty} f(n)=\infty$. Using Assumption \ref{assume0}, the condition \eqref{assume2}, and the change of variables,  for  sufficiently large $n$,
\begin{align}\label{342}
g_2(a_1,\cdots,a_{k-1})- u_n=\int_{f(n)}^{\infty} \phi_k d\nu \leq \int_{f(n)}^{\infty} \phi_k e^{(p_m+\theta)\phi_m} dx \nonumber \\
\leq \int_{n^\alpha}^{\infty} ye^{(p_m+2\theta)y^{\gamma}}y^M dy \leq C\exp((p_m+3\theta)n^{\alpha \gamma}).
\end{align}
 Define the event $E^1_n$ by
\begin{align*}
E^1_{n} := \{|\sum_{i=1}^n (\phi_k(X_i) \1_{X_i\leq f(n)}-u_n)|>n^{\beta}\}.
\end{align*}
Since $0\leq \phi_k(X_i) \1_{X_i\leq f(n)} \leq n^\alpha$, according to the Hoeffding's inequality \cite{h},
\begin{align}\label{343}
\mathbb{Q}(E^1_{n} ) \leq 2\exp(-2n^{2\beta-1-2\alpha }).
\end{align}
Now, define $E^2_{n}$ to be the event for which there exists the set of indices $I$ satisfying $|I| = h(n):= [n^{\gamma-\frac{\alpha \gamma}{2}}]$ such that $X_i>f(n)$ for all $i\in I$. Then, using Assumption \ref{assume0} and the change of variables, for  sufficiently large $n$,
\begin{align} \label{344}
\mathbb{Q}(E^2_{n} ) &<  {n\choose h(n)}\Big[\int_{f(n)}^\infty e^{(p_m+\theta)\phi_m} dx\Big]^{h(n)} \nonumber \\
& < Cn^{h(n)} \Big[\int_{n^{\alpha }} ^\infty e^{(p_m+2\theta)y^\gamma} y^M dy\Big]^{h(n)} <C\exp \big[C(p_m+3\theta)n^{\gamma +\frac{\alpha  \gamma}{2}}\big].
\end{align}
Finally, let us fix a constant $\eta>0$  satisfying
\begin{align} \label{eta}
s-\epsilon <\big(\frac{s}{(s+\eta)^{\gamma}}\big)^{\frac{1}{1-\gamma}},
\end{align}
and then define $E^3_{n}$ to be the event for which $\sum_{i\in I} \phi_m(X_i) >(s+\eta)^{\gamma}n^{\gamma}$ for some $I$ satisfying $|I|<h(n)$. 
Using the result \eqref{303} in Lemma \ref{lemma 3.1}, for sufficiently large $n$,
\begin{align} \label{345}
\mathbb{Q}(E^3_{n} ) &<  C{n\choose h(n)}\big[(s+\eta)^{\gamma}n^{\gamma}-Ch(n)\big]^{h(n)-1} \exp\big[(p_m+\theta)((s+\eta)^{\gamma}n^{\gamma}-Ch(n))\big]\nonumber  \\
&<C\exp \big[(p_m+2\theta)((s+\eta)^{\gamma}n^{\gamma}-Ch(n))\big].
\end{align}
Now, let us check that
\begin{align} \label{346}
(E^1_{n})^c \cap (E^2_{n})^c \cap (E^3_{n})^c \cap C^{\delta}_n \Rightarrow \big\{M_n >  (\frac{s-2\delta}{(s+\eta)^{\gamma}})^{\frac{1}{1-\gamma}} \big\}\cap  C^{\delta}_n.
\end{align}
If we define $I:= \{1\leq i\leq n | X_i>f(n)\}$, then $(E^2_{n})^c \cap (E^3_{n})^c$ imply $|I|<h(n)$ and 
\begin{align} \label{3488}
\sum_{i\in I} \phi_m(X_i) \leq (s+\eta)^{\gamma}n^{\gamma}.
\end{align}
Under the event $(E^1_{n})^c \cap C^{\delta}_n$,
\begin{align*} 
|\sum_{i\in I} \phi_k(X_i) - (a_k-u_n)n| < \delta n+ n^\beta.
\end{align*}
Combining this with \eqref{342}, we obtain
\begin{align} \label{349}
|\sum_{i\in I} \phi_k(X_i) - sn| &= |\sum_{i\in I} \phi_k(X_i) - (a_k-g_2(a_1,\cdots,a_{k-1}))n| \nonumber  \\
 &< \delta n+ n^\beta+C\exp(C(p_m+3\theta)n^{\alpha \gamma})=:r(n).
\end{align}
Thus, using \eqref{3488}, \eqref{349}, we have
\begin{align*}
sn-r(n)<\sum_{i\in I} \phi_k(X_i) &\leq   \Big[\max_{i\in I} \frac{\phi_k(X_i)}{\phi_m(X_i)}\Big] \cdot \sum_{i\in I} \phi_m(X_i) \leq  \Big[\max_{i\in I} \frac{\phi_k(X_i)}{\phi_m(X_i)}\Big]\cdot(s+\eta)^{\gamma}n^{\gamma},
\end{align*}
which implies that for some index $i$,
\begin{align*}
\frac{\phi_k(X_i)}{\phi_m(X_i)} \geq \frac{sn-r(n)}{(s+\eta)^\gamma n^\gamma}.
\end{align*}
Thus, combining this with the condition \eqref{assume2}, for sufficiently large $n$, 
\begin{align*}
M_n^{1-\gamma} \geq \frac{s-2\delta}{(s+\eta)^{\gamma}},
\end{align*}
since $\lim_{n\rightarrow \infty} \frac{r(n)}{n}=\delta$ (recall that $p_m+3\theta<0$).
This concludes the proof of \eqref{346}.

Therefore, using \eqref{343}, \eqref{344}, \eqref{345}, and \eqref{346}, for each $\delta>0$,
\begin{align} \label{347}
\limsup_{n\rightarrow \infty} &\frac{1}{n^{\gamma}}\log \mathbb{Q}\Big(\big\{M_n<(\frac{s-2\delta}{(s+\eta)^{\gamma}})^{\frac{1}{1-\gamma}}\big\} \cap C_n^\delta\Big) \nonumber \\
&\leq \limsup_{n\rightarrow \infty} \frac{1}{n^{\gamma}}\log \mathbb{Q}(E^1_n \cup E^2_n \cup E^3_n) \leq (p_m+2\theta)(s+\eta)^{\gamma}
\end{align} 
(recall that due to the condition \eqref{340}, $2\beta-1-2\alpha > \gamma$).
Note that due to the condition \eqref{eta}, for sufficiently small $\delta>0$,
\begin{align*}
s-\epsilon < \big(\frac{s-2\delta}{(s+\eta)^{\gamma}}\big)^{\frac{1}{1-\gamma}}.
\end{align*}
Thus, using \eqref{341} and \eqref{347}, for such $\eta>0$,
\begin{align}\label{348}
&\limsup_{\delta \rightarrow 0}\limsup_{n\rightarrow \infty} \frac{1}{n^{\gamma }}\log \mathbb{Q}(M_n<s-\epsilon |C_n^{\delta})  \nonumber\\
&\leq
\limsup_{\delta \rightarrow 0}\limsup_{n\rightarrow \infty} \frac{1}{n^{\gamma}}\log \mathbb{Q}\Big(\big\{M_n<\big(\frac{s-2\delta}{(s+\eta)^{\gamma}}\big)^{\frac{1}{1-\gamma}}\big\} \cap C_n^{\delta}\Big) - \liminf_{\delta \rightarrow 0}\liminf_{n\rightarrow \infty} \frac{1}{n^{\gamma}}\log \mathbb{Q}(C_n^{\delta}) \nonumber \\
&\leq (p_m+2\theta)(s+\eta)^{\gamma}-p_m s^{\gamma}.
\end{align}
Since for sufficiently small $\theta>0
$, $(p_m+2\theta)(s+\eta)^{\gamma}-p_ms^{\gamma}<0$ (recall that $p_m+2\theta<0$), proof of \eqref{141} is concluded.

Now, let us prove \eqref{142}. Recall that we have the upper tail estimate:
\begin{align} \label{upper2}
\limsup_{\delta \rightarrow 0} &\limsup_{n\rightarrow \infty} \frac{1}{n}\log \mathbb{Q}(M_n \geq a_k-g_2(a_1,\cdots,a_{k-1})+\epsilon | C^\delta_n) < 0
\end{align} 
according to Theorem  \ref{theorem 1.3}. Indeed, changing the reference measure from $\mathbb{P}$ to $\mathbb{Q}$ does not affect the estimate \eqref{upper2} due to the observation Remark \ref{remark 3.2}.
Combining \eqref{upper2} with \eqref{141}, \eqref{142} immediately follows. 

\end{proof}

Theorem \ref{theorem 1.4} claims that when $(a_1,\cdots,a_{k-1})\in \mathcal{S}_1$ and $a_k>g_2(a_1,\cdots,a_{k-1})$, localization happens in the sense that \eqref{142} holds.
One can also show that localization only happens at the single site. Let us denote $N_n$ by the second largest component among $\frac{\phi_k(X_i)}{n}$'s, and  prove that $N_n$ gets closer to zero in the following sense:

\begin{theorem}\label{theorem 3.6}
Under the same condition as in Theorem \ref{theorem 1.4}, for any $\epsilon>0$,
\begin{align} \label{360}
\lim_{\delta \rightarrow 0}\lim_{n\rightarrow \infty} \mathbb{Q}(\{|M_n-(a_k-g_2(a_1,\cdots,a_{k-1}))|<\epsilon\} \cap \{|N_n|<\epsilon\}|C_n^{\delta})=1.
\end{align}
\end{theorem}
\begin{proof} Throughout the proof, we use the notation $s:=a_k-g_2(a_1,\cdots,a_{k-1})$ and
\begin{align*}
S^i_{n-2} := \frac{\phi_i(X_{3})+\cdots+\phi_i(X_n)}{n-2}
\end{align*}
for each $1\leq i\leq k$. In order to prove \eqref{360}, it suffices to prove that for any $\epsilon>0$,
\begin{align*}
\lim_{\delta \rightarrow 0}\lim_{n\rightarrow \infty} \mathbb{Q}(\{|M_n-s|<\epsilon\} \cap \{|N_n|\leq 2\epsilon\}|C_n^{\delta})=1.
\end{align*}
Thanks to Theorem \ref{theorem 1.4}, it reduces to show that
\begin{align*}
\lim_{\delta \rightarrow 0} \lim_{n\rightarrow \infty} \mathbb{Q}(\{|M_n-s|<\epsilon\} \cap \{|N_n|>2\epsilon\}|C_n^{\delta})=0.
\end{align*}
Thus, proof is concluded once we show the stronger statement:
\begin{align} \label{366}
\limsup_{\delta \rightarrow 0}\limsup_{n\rightarrow \infty} \frac{1}{n}\log \mathbb{Q}(\{|M_n-s|<\epsilon\} \cap \{|N_n|>2\epsilon\}|C_n^{\delta})<0.
\end{align}
According to Remark \ref{remark 3.2}, it suffices to prove the estimate \eqref{366} under the reference measure $\mathbb{P}$ instead of $\mathbb{Q}$.
One can take closed sets $F_n^\delta$ such that for sufficiently large $n$,
\begin{align*}
F_n^\delta \subset \prod_{i=1}^{k-1} [a_i-2\delta,a_i+2\delta]\times [0,a_k-s-\epsilon+2\delta],
\end{align*}
\begin{align*}
\Big\{\Big \vert \frac{\phi_k(X_1)}{n}-s \Big \vert <\epsilon\Big\} \bigcap \Big\{\frac{\phi_k(X_2)}{n} > 2\epsilon\Big\} \bigcap C^\delta_n \Rightarrow (S^1_{n-2},\cdots,S^k_{n-2})\in F_n^\delta.
\end{align*}
Since the sequence of empirical means $(S^1_n,\cdots,S^k_n)$ satisfy the weak LDP with a rate function $I$, we have
\begin{align} \label{361}
\liminf_{n\rightarrow \infty} \frac{1}{n}\log \mathbb{P}(C_n^\delta) \geq -\inf_{v_i\in (a_i-\frac{\delta}{2},a_i+\frac{\delta}{2})} I(v_1,\cdots,v_k) \geq -I(a_1,\cdots,a_k),
\end{align}
\begin{align} \label{362}
\limsup_{n\rightarrow \infty} &\frac{1}{n}\log \mathbb{P}((S^1_{n-2},\cdots,S^k_{n-2})\in F_n^\delta) \nonumber \\
&\leq -\inf_{v_1\in [a_1-2\delta,a_1+2\delta],\cdots, v_{k-1}\in [a_{k-1}-2\delta,a_{k-1}+2\delta],v_k \in [0,a_k-s-\epsilon+2\delta]} I(v_1,\cdots,v_k).
\end{align}
Sending $\delta\rightarrow 0$, using  \cite[Lemma 4.1.6]{dz} and Proposition \ref{prop 2.7},
\begin{align} \label{363}
\lim_{\delta \rightarrow 0}&\inf_{v_1\in [a_1-2\delta,a_1+2\delta],\cdots, v_{k-1}\in [a_{k-1}-2\delta,a_{k-1}+2\delta],v_k \in [0,a_k-s-\epsilon+2\delta]} I(v_1,\cdots,v_k) \nonumber \\
& = \inf_{v_k\in [0,a_k-s-\epsilon]}I(a_1,\cdots,a_{k-1},v_k)  \nonumber \\
&= I(a_1,\cdots,a_{k-1},g_2(a_1,\cdots,a_{k-1})-\epsilon) > I(a_1,\cdots,a_{k-1},a_k).
\end{align}
Therefore, using  \eqref{361}, \eqref{362}, and \eqref{363}, 
\begin{align} \label{364}
&\limsup_{\delta \rightarrow 0} \limsup_{n\rightarrow \infty} \frac{1}{n}\log \mathbb{P}\Big(\big\{ \big \vert\frac{\phi_k(X_1)}{n}-s \big \vert <\epsilon\big\} \cap \big\{\frac{\phi_k(X_2)}{n} > 2\epsilon\big\}\Big \vert C_n^\delta\Big) \nonumber \\
&\leq \limsup_{\delta \rightarrow 0} \limsup_{n\rightarrow \infty} \frac{1}{n}\log \mathbb{P}((S^1_{n-2},\cdots,S^k_{n-2})\in F_n^\delta) - \liminf_{\delta \rightarrow 0} \liminf_{n\rightarrow \infty} \frac{1}{n}\log \mathbb{P}(C_n^\delta)  < 0.
\end{align}
It is also obvious that
\begin{align*}
\mathbb{P}(\{|M_n-s|<\epsilon\} \cap \{|N_n|>2\epsilon\}|C_n^{\delta}) \leq n^2 \mathbb{P}\Big(\big\{\big \vert\frac{\phi_k(X_1)}{n}-s \big \vert <\epsilon\big\} \cap \big\{\big \vert \frac{\phi_k(X_2)}{n} \big \vert >2\epsilon\big\}\Big \vert C_n^{\delta}\Big).
\end{align*}
Thus, this and \eqref{364} conclude the proof of \eqref{366}.
\end{proof}

\section{Examples} \label{section 5}
In this section, we present some concrete examples of the microcanonical distributions for which the aforementioned theories can be applied. In particular, we establish the principle of equivalence of ensembles, and study the localization and delocalization phenomena.
\subsection{Single constraint} \label{section 5.1}
We first consider the microcanonical ensemble given by a single constraint with an unbounded macroscopic observable. We refer to   \cite[Section 7.3]{dz} for the equivalence of ensembles result for this case. In this section, using the large deviation results obtained in Section \ref{section 3}, we derive the equivalence of ensembles result in a different way. We also prove that localization cannot happen.

Suppose that a function $\phi:(0,\infty)\rightarrow (0,\infty)$ satisfies the conditions (C1) and (C2) in  Assumption \ref{assume0}.
Define $\lambda$ to be a probability measure on $(0,\infty)$ whose  distribution is given by $\frac{1}{Z}e^{-\phi}dx$. The reference measure on the configuration space $(0,\infty)^{\mathbb{N}}$ is given by $\mathbb{P}=\lambda^{\otimes \mathbb{N}}$, and let us   denote  $X_i:\Omega \rightarrow (0,\infty)$ by the projection onto the $i$-th coordinate.
Let us consider the microcanonical ensemble
\begin{align*}
\mathbb{P}((X_1,\cdots,X_n)\in \cdot \ |  \ C_n^\delta),
\end{align*}
 where the constraint is given by
\begin{align} \label{40}
C^\delta_n := \Big\{ \big \vert \frac{\phi(X_1)+\cdots+\phi(X_n)}{n}-a\big \vert \leq \delta\Big\}.
\end{align}
We define $S_n:= \frac{\phi(X_1)+\cdots+\phi(X_n)}{n}$ and $H(p):=\log \int e^{p\phi} d\lambda$. Note that $H(p)<\infty$ if and only if $p<1$, and $H$ is differentiable on the interval $(-\infty,1)$. Thanks to the Cram\'er's theorem, the sequence $S_n$ under the reference measure $\mathbb{P}$ satisfies the (full) LDP with a good rate function $I$ which is the Legendre transform of $H$.

Throughout this section, we assume that $a$ belongs to the image of $(-\infty,1)$ under the map $H'$  in order that the conditional distribution is well-defined. In fact, if $a=H'(p)$ for some $p\in (-\infty,1)$, then $p\in \partial I(a)$, which implies that $I(a)<\infty$. We first derive the equivalence of ensembles result.

\begin{proposition} \label{prop 5.1}
For any fixed positive integer $j$,
\begin{align}\label{411}
\lim_{\delta \rightarrow 0} \lim_{n\rightarrow \infty} \mathbb{P}((X_1,\cdots,X_j)\in \cdot   |  C^\delta_n) = (\lambda^*)^{\otimes j}.
\end{align}
Here $\lambda^*$ is a probability measure on $(0,\infty)$ whose distribution is given by $\frac{1}{Z}e^{p\phi}d\lambda$ for $p\in (-\infty,1)$ satisfying
$H'(p)=a$, or equivalently $\int \phi d\lambda^* = a$. 
\end{proposition}
\begin{proof}
Uniqueness of $p\in (-\infty,1)$ satisfying $H'(p)=a$ is obvious since $H$ is strictly convex on $(-\infty,1)$.  According to the LDP result for the single constraint  case (see Remark \ref{remark 2.3}) and the Gibbs conditioning principle, \eqref{411} holds for $\lambda^*$ which is a unique minimizer of
\begin{align} \label{412}
\mu \mapsto H(\mu | \lambda) + a-\int \phi d\mu
\end{align} 
over the constraint $\int \phi d\mu \leq a$.  For any $\mu\ll dx$ with $\int \phi d\mu \leq a$,
\begin{align} \label{413}
 H(\mu | \lambda)+ a-\int \phi d\mu &= H(\mu | \lambda^*) +p\int \phi d\mu +a-\int \phi d\mu +C  \nonumber \\
 &\geq a-(1-p)a+C
\end{align}
 for some universal constant $C$.
Also, equality holds if and only if $\mu = \lambda^*$ since $\int \phi d\lambda^* = a$. Thus, the infimum of \eqref{412} is uniquely obtained at $\mu = \lambda^*$. This concludes the proof.

Note that in the view of \eqref{413}, since 
\begin{align*}
H(\mu | \lambda) + a-\int \phi d\mu = -h(\mu) +a,
\end{align*}  one can also check that
\begin{align} \label{single1}
\lambda^* =  \argmax_{\int \phi d\mu \leq a} h(\mu) = \argmax_{\int \phi d\mu = a} h(\mu) .
\end{align}
\end{proof}

As in Remark \ref{uniform}, Proposition \eqref{prop 5.1}   holds under the uniform distribution on the constraint $C^\delta_n$ as well. Proposition \eqref{prop 5.1} claims that in the equivalence of ensembles viewpoint, when we consider the uniform distribution on the   single constraint \eqref{40} with an unbounded  function $\phi$, it behaves similarly to the case when $\phi$ is bounded (see Theorem \ref{equivalence} for bounded $\phi$). This is a striking difference from the multiple constraints case we have discussed so far.

Now, we show that localization cannot happen when the  microcanonical ensemble is given by a  single constraint \eqref{40}.
\begin{proposition}  \label{prop 5.2}
For any $\epsilon>0$,
\begin{align*}
\limsup_{\delta \rightarrow 0} \limsup_{n\rightarrow \infty} \frac{1}{n}\log \mathbb{P}(M_n \geq \epsilon | C_n^\delta) <0.
\end{align*}
In particular, localization does not happen in the sense that
\begin{align*}
\lim_{\delta \rightarrow 0} \lim_{n\rightarrow \infty} \mathbb{P}(M_n < \epsilon | C_n^\delta) = 1.
\end{align*}
\end{proposition}
\begin{proof} 
Let us choose a reference measure $\nu = \frac{1}{Z}e^{c\phi}d\lambda$ for $c<1$ such that
\begin{align*}
\int \phi d\nu > a.
\end{align*}
In fact, such $c$ exists since $\lim_{p\rightarrow 1^-}H(p)=\infty$ and $H$ is strictly convex.
According to the Cram\'er's theorem, the sequence $S_n$ under the new reference measure $\mathbb{Q}:=\nu^{\otimes \mathbb{N}}$ satisfies the (full) LDP with a good rate function $\bar{I}(v)$ which is the  Legendre transform of $\bar{H}(p) = \log \int e^{px}d\nu(x)$. Thus, for each $\delta>0$,
\begin{align}\label{421}
\liminf_{n\rightarrow \infty} \frac{1}{n}\log \mathbb{Q}(C^\delta_n) \geq -\inf_{v\in (a-\delta,a+\delta)}\bar{I}(v) \geq -\bar{I}(a).
\end{align}
For sufficiently large $n$, we have
\begin{align*}
\Big\{\frac{\phi(X_1)}{n} \in [\epsilon,a]\Big\} \cap C^\delta_n \Rightarrow S_{n-1}:=\frac{\phi(X_2)+\cdots+\phi(X_n)}{n-1} \in [0,a-\epsilon+2\delta].
\end{align*}
Using the fact that
\begin{align*}
\mathbb{Q} (M_n\in [\epsilon,a])\leq n \mathbb{Q}\Big(\frac{\phi(X_1)}{n}\in [\epsilon,a]\Big),
\end{align*}
we have
\begin{align}\label{422}
\limsup_{n\rightarrow \infty} \frac{1}{n}\log \mathbb{Q}(\{M_n \in [\epsilon,a]\} \cap C^\delta_n) &\leq \limsup_{n\rightarrow \infty} \frac{1}{n}\log \mathbb{Q}\Big(\Big\{\frac{\phi(X_1)}{n} \in [\epsilon,a]\Big\} \cap C^\delta_n\Big) \nonumber \\
&\leq \limsup_{n\rightarrow \infty} \frac{1}{n}\log \mathbb{Q}( S_{n-1}\in [0,a-\epsilon+2\delta]) \nonumber \\
&\leq -\inf_{v\in [0,a-\epsilon+2\delta]}\bar{I}(v).
\end{align}
Sending $\delta \rightarrow 0$, applying  \cite[Lemma 4.1.6]{dz}, we have
\begin{align} \label{423}
\lim_{\delta \rightarrow 0} \inf_{v\in [0,a-\epsilon+2\delta]}\bar{I}(v) = \inf_{v\in [0,a-\epsilon]}\bar{I}(v).
\end{align}
Now, let us prove that 
\begin{align} \label{424}
 \inf_{v\in [0,a-\epsilon]}\bar{I}(v) > \bar{I}(a).
\end{align}
According to   \cite[Lemma 2.2.5]{dz}, $\bar{I}$ is non-increasing on the interval $(0,\int \phi d\nu)$. Since $\int \phi d\nu >a$, this implies that $\inf_{v\in [0,a-\epsilon]}\bar{I}(v) = \bar{I}(a-\epsilon)$. If $\bar{I}(a-\epsilon) = \bar{I}(a)$, then $\bar{I}(v)=\bar{I}(a)$ for all $v\in (a-\epsilon,a)$, which means that $\bar{I}'(v)=0$. Thus, $v\in \partial \bar{H}(0)$ for all $v\in (a-\epsilon,a)$, which leads to the contradiction since $\bar{H}$ is differentiable at 0. Thus,  $\bar{I}(a-\epsilon) \neq \bar{I}(a)$, and since $\bar{I}$ is non-increasing on the interval $(0,\int \phi d\nu)$, \eqref{424} is proved.

  Therefore, using \eqref{421}, \eqref{422}, \eqref{423}, and \eqref{424},
\begin{align*}
\limsup_{\delta \rightarrow 0} &\limsup_{n\rightarrow \infty} \frac{1}{n}\log \mathbb{Q}(M_n \in [\epsilon,a] | C^\delta_n) \\
&\leq \limsup_{\delta \rightarrow 0}\limsup_{n\rightarrow \infty} \frac{1}{n}\log \mathbb{Q}(\{M_n \in [\epsilon,a]\} \cap C^\delta_n) - \liminf_{\delta \rightarrow 0}\liminf_{n\rightarrow \infty} \frac{1}{n}\log \mathbb{Q}(C^\delta_n)  \\
& \leq -\inf_{v\in [0,a-\epsilon]}\bar{I}(v) + \bar{I}(a) <0.
\end{align*}
Note that according to \eqref{321},
\begin{align*}
\limsup_{\delta \rightarrow 0} \limsup_{n\rightarrow \infty} \frac{1}{n}\log \mathbb{P}(A  | C^\delta_n) = \limsup_{\delta \rightarrow 0} \limsup_{n\rightarrow \infty} \frac{1}{n}\log \mathbb{Q}(A  | C^\delta_n)
\end{align*}
for any Borel set $A$. Therefore, the proof is concluded.
\end{proof}

\subsection{Two constraints: $l^p$ spheres} \label{section 5.2} In this section, we consider the microcanonical distribution given by two $l^p$-constraints. In particular, we consider the case $\phi_1(x)=x$ and $\phi_2(x)=x^2$.
This type of the microcanonical ensemble was previously studied by Chatterjee \cite{cha2}. He established the convergence of finite marginal distributions and the localization phenomenon. However, the approach  used in \cite{cha2} is ad hoc and only adapted to the special case, so in this section we obtain the result using the unifying theory developed throughout this paper.

It is obvious that $\phi_1(x)=x$ and $\phi_2(x)=x^2$ satisfy Assumption \ref{assume0}. Note that the reference measure $\mathbb{P}$ on the configuration space $(0,\infty)^{\mathbb{N}}$ is given by  $\mathbb{P}=\text{exp}(1)^{\otimes \mathbb{N}}$. Since $\int x^2 d\mu \geq (\int x d\mu)^2$ for any $\mu\in \mathcal{M}_1(\R^+)$ and 
\begin{align*}
\Big\{\mu\in \mathcal{M}_1(\R^+) \Big \vert \mu \ll dx, \int xd\mu = v_1, \int x^2 d\mu \leq v_2\Big\}
\end{align*} 
is a non-empty set whenever $v_2> v_1^2$,  we have $g_1(v_1)=v_1^2$. This means that $\mathcal{A}_1 = \R^+$, and the admissible set is defined by 
\begin{align*}
\mathcal{A}=\{(v_1,v_2)\in (0,\infty)^2 | v_2 > v_1^2\}.
\end{align*} 
We have $v_1\in \pi_1(\partial H(p_1,0))$ for $p_1$ satisfying
\begin{align*}
\frac{\int xe^{p_1x}d\lambda}{\int e^{p_1x}d\lambda}=v_1
\end{align*} 
(see Definition \ref{def} for the meaning of projection $\pi_1$). For such $p_1$ ($p_1=1-\frac{1}{v_1}$), one can check that $\partial H(p_1,0) = \{(v_1,v_2)\in (0,\infty)^2 | v_2\geq 2v_1^2\}$ using the fact that
\begin{align*}
\frac{\int x^2e^{p_1x}d\lambda}{\int e^{p_1x}d\lambda} = 2v_1^2.
\end{align*} 
Thus, $g_2$ can be chosen as $g_2(v_1)=2v_1^2$. Obviously, $\mathcal{S}_1 = \R^+$ and $\mathcal{S}_2$ is an empty set. Also,   according to Proposition \ref{prop 2.7}, a weak LDP rate function $I$ for the sequence $(S^1_n,S^2_n)$ satisfies that for any $c>0$,
\begin{align} \label{dim2}
I(v_1,2v_1^2) = I(v_1,2v_1^2+c).
\end{align}
 For $r^2< s< 2r^2$, define $G_{r,s}$ by a probability measure on $(0,\infty)$ whose  distribution is of the form $\frac{1}{Z_{r,s}}e^{\alpha x+\beta x^2}dx$ and  satisfying
\begin{align} \label{430}
\int x dG_{r,s} = r, \quad \int x^2dG_{r,s}=s.
\end{align}
The existence of such measure can be deduced from Proposition \ref{theorem 2.10}. We first derive the following equivalence of ensembles result as an application of Theorem \ref{theorem 1.0}.
\begin{proposition}
Fix any positive integer $j$.
In the case of $a_1^2<a_2<2a_1^2$, 
\begin{align*}
\lim_{\delta \rightarrow 0} \lim_{n\rightarrow \infty} \mathbb{P}((X_1,\cdots,X_j)\in \cdot  | C^\delta_n) = G_{a_1,a_2}^{\otimes j}.
\end{align*}
On the other hand, in the case of $a_2 \geq 2a_1^2$, 
\begin{align*}
\lim_{\delta \rightarrow 0} \lim_{n\rightarrow \infty} \mathbb{P}((X_1,\cdots,X_j)\in \cdot  | C^\delta_n) =  \text{exp} (a_1)^{\otimes j}.
\end{align*}
\end{proposition}
 
 Finally, let us derive the localization and delocalization result. Let us denote $M_n$ by the maximum component $M_n:= \max_i \frac{X_i^2}{n}$.  Since $\int x d\lambda = 1$ and $\phi_1,\phi_2$ satisfy the condition \eqref{assume2}, the results of Theorem \ref{theorem 1.3} and \ref{theorem 1.4} read as follows:
  \begin{proposition}
Suppose that $a_1=1$, and fix any $\epsilon>0$. In the case of $1<a_2\leq 2$,  localization does not happen in the sense that
\begin{align*}
\lim_{\delta \rightarrow 0} \lim_{n\rightarrow \infty}  \mathbb{P}(M_n > \epsilon | C^\delta_n) = 0.
\end{align*}
On the other hand, in the case of $a_2>2$, localization happens in the sense that
\begin{align*}
\lim_{\delta \rightarrow 0} \lim_{n\rightarrow \infty}  \mathbb{P}(|M_n-(a_2-2)| > \epsilon  | C^\delta_n ) = 0.
\end{align*}
\end{proposition}
Note that when $a_2>2$, the upper tail estimate for  $M_n$ \eqref{upper} reads as
\begin{align*}
\limsup_{\delta \rightarrow 0} &\limsup_{n\rightarrow \infty} \frac{1}{n}\log \mathbb{P}(M_n \geq a_2-2+\epsilon \ | \ C^\delta_n) < 0,
\end{align*}
and the lower tail estimate for $M_n$ \eqref{141} reads as
\begin{align*}
\limsup_{\delta \rightarrow 0} &\limsup_{n\rightarrow \infty} \frac{1}{\sqrt{n}}\log \mathbb{P}(M_n < a_2-2-\epsilon \ | \ C^\delta_n) < 0
\end{align*}
since $\gamma_1=\frac{1}{2}$. As explained in Section \ref{section 2}, the maximum component $M_n$ behaves differently in the upper tail and lower tail regime.

\subsection{Three constraints: $l^p$ spheres.} \label{section 5.3} The last example we consider is the microcanonical ensemble given by three $l^p$-constraints. In particular, we assume that $\phi_i(x)=x^i$ for $i=1,2,3$. It is obvious that these functions satisfy Assumption \ref{assume0}. Note that the reference measure $\mathbb{P}$ on the configuration space $(0,\infty)^{\mathbb{N}}$ is given by  $\mathbb{P}=\text{exp}(1)^{\otimes \mathbb{N}}$. It is not hard to check that $\mathcal{A}_1=\{(v_1,v_2)\in (0,\infty)^2 | v_1^2<v_2\}$, $g_1(v_1,v_2)=\frac{v_2^2}{v_1}$, and the admissible set $\mathcal{A}$ is given by
\begin{align*}
\mathcal{A}=\{(v_1,v_2,v_3)\in (0,\infty)^3 | v_1^2<v_2,\ v_2^2<v_1v_3\}.
\end{align*} We first characterize the sets $S_1$ and $S_2$.

\begin{lemma} \label{lemma 5.5}
The sets $S_1,S_2$ are given by
\begin{align*}
S_1=\{(v_1,v_2)\in (0,\infty)^2 | v_1^2<v_2\leq 2v_1^2\},\  S_2=\{(v_1,v_2)\in (0,\infty)^2 | 2v_1^2<v_2\}.
\end{align*}
\end{lemma}
\begin{proof}
We first  prove the statement for $S_1$.\\
Step 1. If $v_1^2<v_2\leq 2v_1^2$, then there exist $p_1,p_2,v_3$ such that $(v_1,v_2,v_3)\in \partial H(p_1,p_2,0)$: first, we claim that there exist $p_1,p_2$ satisfying that for $i=1,2$,
\begin{align*}
v_i  = \frac{1}{Z}\int x^ie^{p_1x+p_2x^2} d \lambda
\end{align*}
($Z$ is a normalizing constant $Z=\int e^{p_!x+p_2x^2}d\lambda$). In fact, when $v_1^2<v_2<2v_1^2$, this is proved in Section \ref{section 5.2}, and when $v_2=2v_1^2$, one can choose $p_1=1-\frac{1}{v_1},\ p_2=0$. Therefore, for $g_2(v_1,v_2)$ defined by
\begin{align*}
g_2(v_1,v_2)=\frac{1}{Z}\int x^3e^{p_1x+p_2x^2} d\lambda,
\end{align*}
we have
\begin{align*}
\partial H(p_1,p_2,0) = \{(v_1,v_2,w) | w\geq g_2(v_1,v_2)\}
\end{align*}
according to Lemma \ref{prop 2.9}.
This concludes the proof of Step 1.

Step 2. If $2v_1^2<v_2$, then there does not exist $p_1,p_2,v_3$ such that $(v_1,v_2,v_3)\in \partial H(p_1,p_2,0)$: suppose that such $p_1,p_2,v_3$ exist. Then, by Lemma \ref{prop 2.9}, for $i=1,2$,
\begin{align*}
v_i=\frac{1}{Z}\int x^i e^{p_1x+p_2x^2}d\lambda.
\end{align*}
We have $p_2<0$ since $v_2=2v_1^2$ if $p_2=0$. This implies that the logarithmic moment generating function
\begin{align*}
H_1(p_1,p_2) = \log \int e^{p_1x+p_2x^2} d\lambda
\end{align*}
is differentiable at $(p_1,p_2)$, and $(v_1,v_2)\in \partial H_1(p_1,p_2)$. By the Legendre duality, $(p_1,p_2)\in \partial I_1(v_1,v_2)$, where $I_1$ is a Legendre dual of $H_1$. However, due to \eqref{dim2}, $p_2=0$ since $v_2>2v_1^2$, which leads to the contradiction.

The statement for $S_2$ is obvious since $S_2=\mathcal{A}_1 \cap S_1^c$.
\end{proof}

As an application of Theorem \ref{theorem 1.0} and Lemma \ref{lemma 5.5}, we can deduce the following equivalence of ensembles result:

\begin{proposition}
Fix any positive integer $j$. Then,  \begin{align*}
\lim_{\delta \rightarrow 0} \lim_{n \rightarrow \infty} \mathbb{P}((X_1,\cdots,X_j)\in \cdot \ | \ C^\delta_n) = (\lambda^*)^{\otimes j},
\end{align*}
where $\lambda^*$ is characterized as follows: in the case of $a_1^2<a_2\leq 2a_1^2$ and $a_3 \geq g_2(a_1,a_2)$,
\begin{align*}
\lambda^* = \frac{1}{Z}e^{p_1x+p_2x^2} dx
\end{align*}
for $p_1,p_2$ satisfying $\int x^i d\lambda^* = a_i$ for $i=1,2$.

On the other hand, either in the case
(i) $2a_1^2<a_2$ or 
(ii) $a_1^2<a_2\leq 2a_1^2$ and $a_3<g_2(a_1,a_2)$,
\begin{align*}
\lambda^*=\frac{1}{Z}e^{p_1x+p_2x^2+p_3x^3} dx
\end{align*}
for $p_1,p_2,p_3$ satisfying $p_3<0$ and  $\int x^i d\lambda^* = a_i$ for $i=1,2,3$.
\end{proposition}

Finally, since $\phi_i$'s satisfy the condition \eqref{assume2}, one can derive the localization and delocalization result as applications of Theorem \ref{theorem 1.3} and \ref{theorem 1.4}.

\begin{proposition}
Suppose that $(a_1,a_2)\in \mathcal{S}_2$. Then,  localization does not happen in the sense that
\begin{align} \label{prop5.3}
\lim_{\delta \rightarrow 0} \lim_{n\rightarrow \infty}  \mathbb{P}(M_n > \epsilon | C^\delta_n) = 0.
\end{align}
On the other hand, assume that  $(a_1,a_2)\in \mathcal{S}_1$. In the case of $a_3\leq g_2(a_1,a_2)$,  localization does not happen in the sense that \eqref{prop5.3} holds. However, in the case of  $a_3> g_2(a_1,a_2)$, under the reference measure $\mathbb{Q}=\nu^{\otimes 3}$ with $\nu$ of the form:
\begin{align} \label{570}
\nu = \frac{1}{Z}e^{p_1x+p_2x^2}dx,
\end{align} satisfying $\int x^i d\nu = a_i$ for $i=1,2$, localization happens in the sense that
\begin{align*}
\lim_{\delta \rightarrow 0} \lim_{n\rightarrow \infty}  \mathbb{Q}(|M_n-(a_3-g_2(a_1,a_2))| > \epsilon  | C^\delta_n ) = 0.
\end{align*}
\end{proposition}

Note that when $a_3>g_2(a_1,a_2)$, the upper tail estimate for  $M_n$ \eqref{upper} reads as
\begin{align*}
\limsup_{\delta \rightarrow 0} &\limsup_{n\rightarrow \infty} \frac{1}{n}\log \mathbb{Q}(M_n \geq a_3-g_2(a_1,a_2)+\epsilon \ | \ C^\delta_n) < 0,
\end{align*}
and the lower tail estimate for $M_n$ \eqref{141} reads as
\begin{align*}
\limsup_{\delta \rightarrow 0} &\limsup_{n\rightarrow \infty} \frac{1}{n^\gamma}\log \mathbb{Q}(M_n < a_3-g_2(a_1,a_2)-\epsilon \ | \ C^\delta_n) < 0.
\end{align*}
Here, $\gamma = \frac{1}{3}$ when $p_2=0$ in the expression \eqref{570}, and $\gamma = \frac{2}{3}$ when $p_2<0$ in the expression \eqref{570}, since $\gamma_i=\frac{i}{3}$ for $i=1,2$.

\appendix
\section{Auxiliary lemma} \label{section a}
We prove the following auxiliary lemma  frequently used in the paper.
\begin{lemma} \label{lemma 3.1}
Suppose that Assumption \ref{assume0} holds. Also, for some $1\leq m\leq k-1$, consider  the probability distribution 
$\nu = \frac{1}{Z}e^{p_1\phi_1+\cdots+p_m\phi_m}dx$    on $(0,\infty)$ with $p_m<0$. Then,
for any number $M\geq 0$ and $\epsilon>0$,
\begin{align}\label{301}
 \lim_{n\rightarrow \infty} \frac{1}{n}\log \nu\Big(\big \vert \frac{\phi_k(X_1)}{n} - M \big \vert <\epsilon\Big) = 0.
\end{align}
Let us denote $\mathbb{Q}$ by the product measure $\mathbb{Q}=\nu^{\otimes \mathbb{N}}$. Then, for any $0<\theta<-p_m$, there exists $C=C(\theta)>0$ such that
\begin{align} \label{303}
\mathbb{Q}\Big(\sum_{i=1}^j \phi_m(X	_i)>  M\Big) <C (M-Cj)^{j-1} \exp \big[(p_m+\theta)(M-Cj)\big]
\end{align}
for any $j\in \mathbb{N}$, $M>Cj+2$.

Furthermore, under the additional condition \eqref{assume2}, 
\begin{align} \label{302}
\liminf_{n\rightarrow \infty} \frac{1}{n^{\gamma_m}} \log \nu\Big(\big \vert \frac{\phi_k(X_1)}{n}-M\big \vert<\epsilon\Big) \geq p_m M^{\gamma_m}.
\end{align}

\end{lemma}

\begin{proof}
Note that due to Assumption \ref{assume0}, for any $\theta>0$, there exists $C=C(\theta)$ such that
\begin{align*}
x>C \Rightarrow (p_m-\theta) \phi_m<p_1\phi_1+\cdots+p_m\phi_m < (p_m+\theta) \phi_m.
\end{align*}
Let us first prove \eqref{301}. Since $m<k$,  thanks to the condition (C3) in Assumption \ref{assume0}, there exists $0<\delta<1$ such that for  sufficiently large $y$,
\begin{align*}
\sum_{i=1}^m p_i\phi_i (\phi_k^{-1}(y)) < (p_m+\delta)y^{1-\delta}.
\end{align*} Thus, using the condition (C4) in Assumption \ref{assume0} and the change of variables, for  sufficiently large $n$, 
\begin{align*}
\nu\Big(\big \vert \frac{\phi_k(X_1)}{n}-M\big \vert<\epsilon\Big) &= \int_{(M-\epsilon)n}^{(M+\epsilon)n} e^{\sum_{i=1}^m p_i\phi_i (\phi_k^{-1}(y))} \frac{1}{\phi_k'(\phi_k^{-1}(y))}dy \\
&<\int_{(M-\epsilon)n}^{(M+\epsilon)n} Ce^{(p_i+\delta)y^{1-\delta}}y^C dy < C\epsilon n  e^{(p_i+\delta)((M+\epsilon)n)^{1-\delta}}((M+\epsilon)n)^C.
\end{align*}
After taking $\log$ and dividing by $n$, and then sending $n\rightarrow \infty$, we obtain \eqref{301}.

Let us now prove \eqref{303}. If we define $Y_i:=\phi_m(X_i)$, then $Y_i$'s are i.i.d. whose individual distribution is given by $\frac{1}{Z}e^{\sum_{i=1}^m p_i\phi_i(\phi_m^{-1}(y))} \frac{1}{\phi_m' (\phi_m^{-1}(y))}dy$ on $(0,\infty)$. Using Assumption \ref{assume0}, for any $0<\theta<-p_m$, there exists $C$ such that 
\begin{align} \label{304}
y>C \Rightarrow \frac{1}{Z}e^{\sum_{i=1}^m p_i\phi_i(\phi_m^{-1}(y))} \frac{1}{\phi_m' (\phi_m^{-1}(y))} < \frac{1}{Z'}e^{(p_m+\theta)y}
\end{align}
($Z'=\int e^{(p_m+\theta)y} dy$ is a normalizing constant).
Let us denote  $Z_1,Z_2,\cdots$ by i.i.d random variables whose individual distribution is given by $\text{exp}(p_m+\theta)$. Then, \eqref{304} implies that for any $K>0$,
\begin{align} \label{305}
\mathbb{Q}\Big(\sum_{i=1}^j \phi_m(X_i) \1_{\phi_m(X_i) \geq C} > K\Big) \leq \mathbb{Q}\Big(\sum_{i=1}^j Z_i >  K\Big)
\end{align}
 by the simple coupling argument.
Using the fact that  law of $\sum_{i=1}^j Z_i $ is $\text{Gamma}(j,p_m+\theta)$, it is easy to check that for $K>2$,
\begin{align} \label{306}
\mathbb{Q}\Big(\sum_{i=1}^j Z_i >  K\Big) <C K^{j-1} e^{(p_m+\theta)K}
\end{align} 
(we refer to \cite{cha2} for the estimate \eqref{306} in the case $\text{Gamma}(j,1)$ distribution). On the other hand, it is obvious that
\begin{align} \label{3066}
\sum_{i=1}^j \phi_m(X_i) > M \Rightarrow \sum_{i=1}^j  \phi_m(X_i) \1_{\phi_m(X_i) \geq C} > M- Cj.
\end{align}  Thus, \eqref{305}, \eqref{306}, and \eqref{3066} conclude the proof of \eqref{303}.

Finally, let us prove \eqref{302} under the additional condition \eqref{assume2}. Using Assumption \ref{assume0}, condition \eqref{assume2}, and the change of variables, for any $\theta,\eta>0$, 
\begin{align*}
\nu\Big(\big \vert \frac{\phi_k(X_1)}{n}-M\big \vert<\eta\Big) &> \int_{(M-\eta)n}^{(M+\eta)n} Ce^{(p_m-\theta)y^{\gamma_m}}y^C dy > C\eta n  e^{(p_m-\theta)((M+\eta)n)^{\gamma_m}}((M-\eta)n)^C 
\end{align*}
for sufficiently large $n$.
After taking $\log$, dividing by $n$,  sending $n\rightarrow \infty$, and then sending $\theta \rightarrow 0$, we have
\begin{align*}
\liminf_{n\rightarrow \infty} \frac{1}{n^{\gamma_m}} \log \nu\Big(\big \vert \frac{\phi_k(X_1)}{n}-M\big \vert<\eta\Big) \geq p_m(M+\eta)^{\gamma_m}.
\end{align*}
Since for $0<\eta<\epsilon$, 
\begin{align*}
\liminf_{n\rightarrow \infty} \frac{1}{n^{\gamma_m}} \log \nu\Big(\big \vert \frac{\phi_k(X_1)}{n}-M\big \vert<\epsilon\Big) >\liminf_{n\rightarrow \infty} \frac{1}{n^{\gamma_m}} \log \nu\Big(\big \vert \frac{\phi_k(X_1)}{n}-M\big \vert<\eta\Big),
\end{align*}
and $\eta>0$ can be arbitrary small, we obtain \eqref{302}.

\end{proof}

\section*{Acknowledgement}
The author thanks to the advisor Fraydoun Rezakhanlou for suggesting this project and sharing many helpful discussions.

\end{document}